\documentclass{amsart}
\usepackage{amscd,amssymb,graphics}

\theoremstyle{plain}
\newtheorem*{maincons}{Main Construction}
\newtheorem{thm}{Theorem}[section]

\newtheorem{lem}[thm]{Lemma}
\newtheorem{cor}[thm]{Corollary}
\newtheorem{cons}[thm]{Construction}
\theoremstyle{remark}
\newtheorem{rem}[thm]{Remark}
\newcommand{\m}{\phantom{-}}
\newcommand{\x}{\negmedspace}
\newcommand{\Z}{\mathbb{Z}}
\newcommand{\R}{\mathbb{R}}
\newcommand{\C}{\mathbb{C}\mkern1mu}
\renewcommand{\H}{\mathbb{H}\mkern1mu}
\newcommand{\Ca}{\mathbb{O}\mkern 1mu}

\newcommand{\CP}{\mathbb{C\mkern1mu P}}

\newcommand{\Sph}{\mathbb{S}}

\DeclareMathOperator{\diag}{diag}

\DeclareMathOperator{\tr}{trace}
\DeclareMathOperator{\id}{id}

\newcommand{\tp}{^{\mathrm{t}}}
\newcommand{\qi}{\mathbf{i}}
\newcommand{\qj}{\mathbf{j}}
\newcommand{\qk}{\mathbf{k}}
\newcommand{\imat}{{\mathchoice
{\mathrm 1\mskip-4.2mu\mathrm l}{\mathrm 1\mskip-4.2mu\mathrm l}
{\mathrm 1\mskip-3.9mu\mathrm l}{\mathrm 1\mskip-4.0mu\mathrm l}}}
\newcommand{\cond}{\;\vert\;}

\newcommand{\bmat}{\left(\begin{smallmatrix}}
\newcommand{\emat}{\end{smallmatrix}\right)}
\newcommand{\bsmat}{\bigl(\begin{smallmatrix}}
\newcommand{\esmat}{\end{smallmatrix}\bigr)}
\newcommand{\Bsmat}{\Bigl(\begin{smallmatrix}}
\newcommand{\Esmat}{\end{smallmatrix}\Bigr)}
\newcommand{\bbsmat}{\biggl(\begin{smallmatrix}}
\newcommand{\eesmat}{\end{smallmatrix}\biggr)}
\newcommand{\BBsmat}{\Biggl(\begin{smallmatrix}}
\newcommand{\EEsmat}{\end{smallmatrix}\Biggr)}
\newcommand{\abs}[1]{\vert #1\vert}
\newcommand{\Abs}[1]{\bigl\vert #1\bigr\vert}

\newcommand{\SO}{\mathrm{SO}}

\newcommand{\SU}{\mathrm{SU}}
\newcommand{\U}{\mathrm{U}}

\newcommand{\Syp}{\mathrm{Sp}}
\newcommand{\Gtwo}{\mathrm{G}_2}
\newcommand{\Spin}{\mathrm{Spin}}
\newcommand{\asyp}{\mathfrak{sp}}
\newcommand{\asu}{\mathfrak{su}}
\newcommand{\aso}{\mathfrak{so}}
\newcommand{\ag}{\mathfrak{g}}
\newcommand{\ah}{\mathfrak{h}}
\newcommand{\ak}{\mathfrak{k}}
\newcommand{\comm}{[\;\cdot\;,\;\cdot\;]}
\DeclareMathOperator{\Real}{Re}
\DeclareMathOperator{\Imag}{Im}

\newcommand{\llangle}{\langle\mkern-3mu\langle}
\newcommand{\rrangle}{\rangle\mkern-3mu\rangle}

\DeclareMathOperator{\power}{\downarrow}
\begin{document}
\title[The commutator of quaternions]{On the commutator of unit quaternions\\ and the numbers $12$ and $24$}
\author{Thomas P\smash{\"u}ttmann}
\address{Ruhr-Universit\"at Bochum\\Fakult\"at f\"ur Mathematik\\
        D-44780 Bochum\\Germany}
\email{Thomas.Puettmann@rub.de}

\begin{abstract}
The quaternions are non-commutative. The deviation from commutativity is encapsulated
in the commutator of unit quaternions. It is known that the $k$-th power of the
commutator is null-homotopic if and only if $k$ is divisible by $12$.
The main purpose of this paper is to construct a concrete null-homotopy of the $12$-th
power of the commutator. Subsequently, we construct free $\mathbb{S}^3$-actions
on $\mathbb{S}^7\times\mathbb{S}^3$ whose quotients are exotic $7$-spheres and
give a geometric explanation for the order 24 of the stable homotopy groups
$\pi_{n+3}(\mathbb{S}^n)$. Intermediate results of perhaps independent interest are a
construction of the octonions emphasizing the inclusion
$\mathrm{SU}(3)\subset \Gtwo$, a detailed study of Duran's geodesic boundary map construction,
and explicit formulas for the characteristic maps of the bundles $\Gtwo\to\Sph^6$ and
$\Spin(7)\to \Sph^7$.
\end{abstract}


\maketitle

%
%

\section{Introduction}
The quaternions form a non-commutative normed division algebra.
Samelson \cite{samelson} and G.\,W.~Whitehead \cite{whitehead} proved that the quaternions are not even homotopically commutative. More precisely, the commutator
\begin{gather*}
  \comm : \Sph^3\times\Sph^3 \to \Sph^3, \quad
  (a,b) \mapsto [a,b] := aba^{-1}b^{-1}
\end{gather*}
of unit quaternions generates the homotopy group $\pi_6(\Sph^3)\approx \Z_{12}$
(see \cite{serre}, \cite{rokhlin}, \cite{james}). It follows that the $k$-th power
\begin{gather*}
  (a,b) \mapsto [a,b]^k =   aba^{-1}b^{-1}aba^{-1}b^{-1}\ldots aba^{-1}b^{-1}
\end{gather*}
is null-homotopic if and only if $k$ is divisible by $12$, and that the map
\begin{gather*}
  (a,b)\mapsto [a,b^k] = ab^k a^{-1}b^{-k}
\end{gather*}
is null-homotopic if and only if $k$ is divisible by $12$ (see \cite{whitehead}).
 
\begin{maincons}
We construct two concrete homotopies
\begin{gather*}
  \Sph^3\times \Sph^3\times [0,1] \to \Sph^3
\end{gather*}
that deform $\comm^{12}$ and $(a,b) \mapsto [a,b^{12}]$, respectively, to the constant map to $1$.
\end{maincons}

\smallskip\enlargethispage{1cm}

The twelve H-space structures on $\Sph^3$ can be represented by the multiplications
$(a,b) \mapsto ab[a,b]^k$ for $k = \{0,\ldots, 11\}$ (see \cite{arkowitz}).
A trivial consequence of our construction is a concrete deformation between
$(a,b) \mapsto ab[a,b]^{12}$ and the standard multiplication on~$\Sph^3$.
Other, less trivial, consequences are based on the following subsequent construction:
Let $\Syp(2)$ denote the group of unitary quaternionic $2\times 2$ matrices and
let $p_{1,2}:\Syp(2) \to \Sph^7$ denote the projection to the first/second column,
respectively. We fix a suitable identification of $\Sph^7$ with the unit octonions.
Let $\power^j$ denote the selfmap of $\Sph^7$ that sends each unit octonion to
its $j$-th power.
\begin{cons}
\label{sptwocons}
Using the null-homotopies of the Main Construction we construct maps
$\chi_j:\Sph^7\to \Syp(2)$ such that $p_1\circ \chi_j = \power^{12j}$.
This yields a concrete identification $\Z \to \pi_7(\Syp(2))$.
\end{cons}

Via the generalized Gromoll-Meyer construction in \cite{dpr} we then obtain concrete exotic free $\Sph^3$-actions on $\Sph^7\times\Sph^3$. The existence of such actions was established by
Hilton and Roitberg in 1968 \cite{hilton}. Our construction owes to the reinvestigation of this
phenomenon in \cite{barros}.

Concretely, let $\chi_{j,2} : \Sph^7 \to \Sph^7$ be an abbreviation for $p_2\circ\chi_j$ and let
$\llangle u,v\rrangle = \bar u\tp v$ denote the standard Hermitian product on the quaternionic vector space $\H^2$.
For a unit quaternion $q \in \Sph^3$, a quaternionic vector $u\in\H^2$, and a unit quaternion $r \in\Sph^3$ set
\begin{gather*}
  q\star_j (u,r) = \bigl( qu\bar q, \llangle \chi_{j,2}(qu\bar q), q \chi_{j,2}(u)\rrangle r\bigr).
\end{gather*}
Note that $q\star_0 (u,r) = \bigl( qu\bar q, qr\bigr)$ since $\chi_0$ is the constant map to the unit matrix in $\Syp(2)$. All other actions $\star_j$ are not isometric with respect to the standard metric on $\Sph^7\times\Sph^3$.
\begin{thm}
\label{exoticaction}
The action $\star_j$ is a free $\Sph^3$-action on $\Sph^7\times\Sph^3$ whose quotient space
is the homotopy $7$-sphere $\Sigma^7_{12j}$ in the notation of \cite{dpr}.
All these homotopy $7$-spheres form a subgroup isomorphic to $\Z_7$ in the group of
orientation preserving diffeomorphism classes of homotopy $7$-spheres
$\Theta_7 \approx \Z_{28}$.
\end{thm}

\smallskip

Another application concerns the stable homotopy groups $\pi_{n+3}(\Sph^n)$, $n\ge 5$.
It is well-known (see e.g. \cite{hu}) that these groups are cyclic of order $24$ generated by the suspensions
$\Sigma^{n-4} h$ of the Hopf map
\begin{gather*}
  h : \Sph^7 \to \Sph^4, \quad \bmat u\\ v\emat \mapsto \bmat \abs{u}^2-\abs{v}^2\\
   2\bar u v\emat.
\end{gather*}
Here, $u$, $v$ are quaternions with $\abs{u}^2+\abs{v}^2 = 1$. 

\begin{cons}
Employing the maps $\chi_j :\Sph^7\to \Syp(2)$ we obtain simple
explicit homotopies between $\Sigma (h \circ \power^{12j})$ and its natural homotopy
inverses $\Sigma^{-1}(h\circ \power^{12j})$ given by reversing the direction of the suspension coordinate.
\end{cons}
In fact, this construction is short enough to perform it right away: The map
\begin{gather*}
  H_j : \Sph^7 \times [0,\tfrac{\pi}{2}] \to \Sph^7, \quad
  H_j (x,t) = \cos t \cdot \chi_{j,1}(x) + \sin t \cdot \chi_{j,2}(x).
\end{gather*}
yields a well-defined homotopy between $\chi_{j,1} = \power^{12j}$ and $\chi_{j,2}$
since corresponding values of $\chi_{j,1}$ and $\chi_{j,2}$ are always perpendicular in $\H^2$.
Now,
\begin{gather}
\label{sptwoid}
  h\bmat c\\ d\emat = \bmat \abs{c}^2 - \abs{d}^2\\ 2 c\bar d\emat
   = \bmat \abs{b}^2 - \abs{a}^2\\ -2a\bar b\emat = - h\bmat a\\ b\emat \quad
   \text{if } \bmat a & c \\ b & d\emat \in \Syp(2).
\end{gather}
Hence, $h \circ H_j$ is a homotopy between $h \circ \power^{12j}$ and $-h\circ\power^{12j}$.
This homotopy induces a homotopy $\Sigma(h \circ H_j)$ between the suspensions
$\Sigma(h\circ \power^{12j})$ and $\Sigma(-h\circ \power^{12j})$.
Now, $\id_{\Sph^5}$ is homotopic to $-\id_{\Sph^5}$ by a block matrix that consists of three
$2\times 2$ rotation matrices. The concatenation of the homotopies deforms
$\Sigma(h\circ \power^{12j}):\Sph^8\to\Sph^5$ to the map
\begin{gather*}
  -\id\circ\Sigma(-h\circ \power^{12j}) = \Sigma^{-1}(h\circ \power^{12j}).
\end{gather*}

\smallskip

The organization of the paper is explained at the end of the next section, which provides
an outline of our Main Construction.

\bigskip

\section{Outline of the Main Construction}
The rough combinatorics of our construction is perhaps not surprising: $12 = 3 \cdot 2 \cdot 2$
where the factor $3$ is geometrically related to the order of the homotopy group $\pi_6(\Gtwo)$,
one factor $2$ comes from killing the double of the fourth suspension of the
Hopf fibration $\Sph^3\to \Sph^2$, and the other factor $2$ from killing the double of the
single suspension of the Hopf fibration $\Sph^3\to \Sph^2$.
The detailed deformations, however, are far from beeing obvious.

\smallskip

The crucial geometric tool is a particular form of the boundary map in the exact homotopy
sequence of fiber bundles. This particular form was first used by Duran~\cite{duran}.
Given a fiber bundle $F\cdots E \to B$ and a map $\alpha: \Sph^k \to B$ one uses
horizontal lifts to define a concrete map $\partial_{E\to B}(\alpha) : \Sph^{k-1}\to F$
that induces the boundary map $\pi_k(B)\to \pi_{k-1}(F)$.
We investigate this construction in connection with suspensions of maps and
powers of spheres and obtain some statements that are specific to this way of lifting.

\smallskip

We now outline our Main Construction. The first step is rather elementary.
\begin{cons}
\label{homeq}
We construct a continuous map $\mu:\Sph^3\times \Sph^3 \to \Sph^6$ and
two homotopies $\Sph^3\times\Sph^3\times [0,1]\to \Sph^3$ that deform
$\comm^{12}$ and $(a,b) \mapsto [a,b^{12}]$, respectively, to the composition
$\bigl(\partial_{\Syp(2)\to\Sph^7}(\id)\bigr)^{12} \circ \mu$.
\end{cons}
Thus it remains to construct a null-homotopy of the $12$-th power of the characteristic map
$\partial_{\Syp(2)\to\Sph^7}(\id)$.
Note that a homotopic rational version of this characteristic map was proven to generate
$\pi_6(\Sph^3)$ by Borel and Serre \cite{serre}.

\smallskip

The central second step is an application of our results about
the horizontal lifting construction in section\,\ref{lifts}.
We apply them to the commutative diagram
\begin{gather*}
  \begin{CD}
    \Sph^3 @>>> \Gtwo @>>>  V_{7,2} \\
     @VVV    @VVV @|\\
    \Syp(2) @>>> \Spin(7) @>>> V_{7,2}\\
    @VVV     @VV{\pi}V\\
    \Sph^7 @= \Sph^7
  \end{CD}
\end{gather*}
Identify $\Sph^7$ with the unit sphere in the octonions $\Ca$ and let $V_{7,2}$
denote the Stiefel manifold of orthonormal $2$-frames in $\R^7 \approx \Imag\Ca$.
Let $e_1, e_2$ denote a fixed orthonormal $2$-frame in $\R^7 \approx \Imag\Ca$.
Set
\begin{gather*}
  \kappa : \Sph^7 \to V_{7,2}, \quad
  a \mapsto (\bar a e_1 a, \bar a e_2 a).
\end{gather*}

\begin{thm}
\label{boundarymaps}
We have
\begin{gather*}
  \bigl(\partial_{\Syp(2)\to\Sph^7}(\id)\bigr)^6
  = \partial_{\Gtwo\to V_{7,2}}(\kappa\circ \power^2),
\end{gather*}
where $\power^2$ denotes the octonionic squaring map.
\end{thm}
This identity is obtained from the identity
$\bigl(\partial_{\Spin(7)\to\Sph^7}(\id)\bigr)^6 = \imat$,
which in turn is intimately related to the geometric presentation of $\pi_6(\Gtwo)\approx \Z_3$
discovered by Chaves and Rigas \cite{rigas}.

\smallskip

The third step is the most technical one. Toda, Saito, and Yokota \cite{saito}
showed that the map $\kappa$ generates the homotopy group $\pi_7(V_{7,2}) \approx \Z_4$
(see also \cite{rigchav} for relations between $\kappa$ and certain Hopf maps).
Note that the exact homotopy sequence of the fibration $\Sph^5\to V_{7,2}\to\Sph^6$
contains the part
\begin{gather*}
  \Z_2 \approx \pi_7(\Sph^5) \to \pi_7(V_{7,2}) \to \pi_7(\Sph^6) \approx \Z_2.
\end{gather*}
Thus, the first column of the map $\kappa\circ\power^2$ is null-homotopic.

\begin{cons}
\label{kappadeform}
We construct a concrete homotopy
\begin{gather*}
  \kappa\circ \power^2 \sim (N,\Sigma^2\tau).
\end{gather*}
Here, $N$ denotes the constant map from $\Sph^7$ to the northpole $e_1$ of $\Sph^6$ and
$\Sigma^2\tau$ is the double suspension of a specific map $\tau: \Sph^5 \to \Sph^3$.
\end{cons}

The construction of this homotopy involves two steps. We first recognize that
the map $\kappa$ actually consists of two perpendicular variants of the fibration
$\Sph^7 \to \CP^3$ composed with the cut locus collapse $\CP^3\to \Sph^6$.
We deform $\kappa$ concretely to a map $h$ that consists of two perpendicular variants
of the fourth suspension of the Hopf fibration $h_1: \Sph^3 \to \Sph^2$.
In the second step we deform $h\circ \power^2$ to the map $(N,\Sigma^2\tau)$.
It is essential here that the second column is the suspension of a map.
The precise form of $\tau$ is actually not important, yet an explicit formula for this map
could be obtained.

\begin{cons}
\label{homlift}
Lifting the deformation curves of the homotopy in Construction\,\ref{kappadeform}
horizontally yields a homotopy
\begin{gather*}
  \partial_{\Gtwo\to V_{7,2}}(\kappa\circ \power^2)
 \sim A_0\cdot \partial_{\Gtwo\to V_{7,2}}(N,\Sigma^2\tau)
\end{gather*}
where $A_0$ is a specific matrix in $\SU(3)\subset \Gtwo$ that identifies the fiber
over the point $(e_1,-e_2)$ with the fiber over the point $(e_1,e_2)$.
\end{cons}

For the intuition of the reader we now supply the following commutative diagram
\begin{gather*}
  \begin{CD}
    \SU(2) @= \Sph^3 \\
     @VVV    @VVV\\
    \SU(3) @>>> \Gtwo @>>> \Sph^6\\
    @VVV     @VV{\pi}V @|\\
    \Sph^5 @>>> V_{7,2} @>>> \Sph^6.
  \end{CD}
\end{gather*}

\begin{lem}
\begin{gather*}
  \partial_{\Gtwo\to V_{7,2}}(N,\Sigma^2\tau)
  = \partial_{\SU(3)\to \Sph^5}(\Sigma^2\tau).
\end{gather*}
\end{lem}

We then apply the Eckmann-Kervaire identity
$\partial(\alpha\circ \Sigma\beta)  = \partial(\alpha)\circ\beta$, which also holds for the
specific way of lifting that we use in this paper (see Lemma\,\ref{susplift}).
\begin{lem}
\begin{gather*}
  \partial_{\SU(3)\to \Sph^5}(\Sigma^2\tau)
  = \partial_{\SU(3)\to\Sph^5}(\id_{\Sph^5}) \circ\Sigma\tau.
\end{gather*}
\end{lem}

The map $\partial_{\SU(3)\to\Sph^5}(\id)$ is the suspension of the
Hopf fibration $\Sph^3 \to \Sph^2$. Null-homotopies of twice this map
are classical. The simplest null-homotopy $H_{\SU(3)\to\Sph^5}$ of
$\bigl(A_0 \cdot \partial_{\SU(3)\to\Sph^5}(\id)\bigr)^2$ in precisely this
form is given in \cite{unitary} (see Theorem\,\ref{suhom}).

\begin{thm}
\label{finalthm}
The concatenation of the null-homotopy $H_{\SU(3)\to\Sph^5}$
with the homotopies of the previous statements provides a null-homotopy
\begin{gather*}
  \bigl(\partial_{\Syp(2)\to\Sph^7}(\id)\bigr)^{12}
  \sim \bigl(A_0\cdot \partial_{\SU(3)\to\Sph^5}(\id)\bigr)^2 \circ\Sigma\tau \sim 1.
\end{gather*}
\end{thm}

The rest of the paper is organized as follows: In section\,\ref{oct} we provide a definition
for the octonionic multiplication based on the complex cross product.
For our concrete computations in the following sections this definition is much
more convenient then the standard definition.
In section\,\ref{lifts} we study the interplay of suspensions, powers of spheres and Duran's
form of the boundary map in the exact homotopy sequence. In section\,\ref{examples} we compute
the characteristic maps of bundles that belong to some transitive actions on spheres
using horizontal lifts.
Moreover, we use the results of section\,\ref{lifts} to prove Theorem\,\ref{boundarymaps}.
In section\,\ref{Scommutator} we perform Construction\,\ref{homeq} and in section\,\ref{kappatwo}
we perform Construction\,\ref{kappadeform} and Construction\,\ref{homlift}.
In section\,\ref{exotic} we perform the Construction\,\ref{sptwocons} and
Theorem\,\ref{exoticaction}.

\bigskip

\section{Octonionic multiplication via the complex cross product}
\label{oct}
We define the {\em complex cross product} of two vectors $z,w\in\C^3$ by
\begin{gather*}
  z\times w = \Bsmat \bar z_2\bar w_3 - \bar z_3\bar w_2\\
    \bar z_3\bar w_1 - \bar z_1\bar w_3\\
    \bar z_1\bar w_2 - \bar z_2\bar w_1\Esmat.
\end{gather*}
Let $\llangle z,w\rrangle = \bar z_1 w_1 + \bar z_2 w_2 + \bar z_3 w_3$ denote
the hermitian inner product on $\C^3$.
Given two vectors $\bsmat z_0\\ z\esmat$, $\bsmat w_0\\ w\esmat \in \C\times \C^3$
we set
\begin{gather*}
  \bsmat z_0\\ z\esmat \cdot \bsmat w_0\\ w\esmat
    := \bsmat z_0 w_0 - \llangle z,w\rrangle \\ \bar z_0 w + w_0 z + z \times w\esmat.
\end{gather*}

\begin{thm}
\label{octonionicmultiplication}
With this product, $\C\times \C^3$ is isomorphic to the octonions.
\end{thm}

The proof will be a consequence of the following properties of the complex cross product, all of which can be verified easily.
\begin{lem}
For all $y,z,w \in \C^3$ we have
\begin{gather*}
  \llangle z, z\times w\rrangle = 0 \quad \text{and} \quad
  y \times (z\times w) = \llangle y,w\rrangle z - \llangle y,z\rrangle w.
\end{gather*}
\end{lem}

\begin{lem}
If $z,w \in \C^3$ are unit vectors with $\llangle z,w\rrangle = 0$  then $z\times w$
is the unique vector in $\C^3$ such that the complex $3\times 3$-matrix
$(z,w,z\times w)$ is contained in $\SU(3)$.
\end{lem}

\begin{cor}
For all $A\in \SU(3)$ and $z,w\in \C^3$ we have
\begin{gather*}
 (A\cdot z)\times (A\cdot w) = A\cdot (z\times w).
\end{gather*}
\end{cor}

\begin{lem}
\label{crossproductnorm}
For all $z,w\in \C^3$ we have
\begin{gather*}
  \abs{z\times w}^2 = \abs{z}^2\abs{w}^2 - \abs{\llangle z,w\rrangle}^2.
\end{gather*}
\end{lem}

\begin{proof}[Proof of Theorem\,\ref{octonionicmultiplication}]
Given an equation of form
\begin{gather*}
  \bsmat a_0\\ a\esmat \cdot \bsmat z_0\\ z\esmat = \bsmat w_0\\ w\esmat
\end{gather*}
we multiply both sides by $\bsmat \bar a_0\\ -a\esmat$ from the left. Using
the previous lemmas we obtain
\begin{gather*}
  (\abs{a_0}^2 + \abs{a}^2)\bsmat z_0\\ z\esmat =
  \bsmat \bar a_0\\ -a\esmat \bsmat w_0 \\ w \esmat.
\end{gather*}
Hence, $\frac{1}{\abs{a_0}^2 + \abs{a}^2}\bsmat \bar a_0\\ -a\esmat$ is the unique
left inverse of $\bsmat a_0\\ a\esmat\neq 0$. Similarly we see that it is also the unique
right inverse. Using Lemma\,\ref{crossproductnorm}
a straightforward computation shows
\begin{gather*}
  \Abs{\bsmat z_0\\ z\esmat \cdot \bsmat w_0\\ w\esmat}^2
  = (\abs{z_0}^2 + \abs{z}^2) (\abs{w_0}^2 + \abs{w}^2).
\end{gather*}
All in all, $\C\times \C^3$ is a normed division algebra.
\end{proof}

\begin{rem}
The two maximal subgroups of the automorphism group $\Gtwo$ of the octonions
are $\SU(3)$ and $\SO(4)$. The inclusion $\SU(3) \subset \Gtwo$ is emphasized by
our construction of the octonions via the complex cross product. 
The inclusion $\SO(4)\subset\Gtwo$, on the other hand, is emphasized in the standard construction of the octonions using pairs of quaternions with the product
\begin{gather*}
  (u_1,v_1)\cdot (u_2,v_2) = (u_1u_2 - \bar v_2 v_1, v_2 u_1 + v_1 \bar u_2).
\end{gather*}
The inclusion $\SO(4)\subset\Gtwo$ is then given by the action
$(q_1,q_2)\cdot (u,v) = (q_1u\bar q_1,q_2 v \bar q_1)$
of two unit quaternions on the octonion $(u,v)$.
A concrete isomorphism between the two realizations of the octonions is given by
\begin{gather*}
  \bsmat z_0\\ z\esmat = \bsmat x_0 + i y_0 \\ x + i y\esmat \mapsto
    (x_0 + x_1 \qi + x_2 \qj + x_3 \qk, y_0 + y_1 \qi + y_2 \qj + y_3 \qk).
\end{gather*}
\end{rem}

\bigskip

\section{Horizontal lifts, suspensions, and powers of spheres}
\label{lifts}
In this section we study the interplay of suspensions, powers of spheres and the
horizontal lifting form of the boundary map in the exact homotopy sequence.

\subsection{Powers of spheres and suspensions}
The {\em $k$-th power of the sphere $\Sph^n \subset \R \times \R^n$} is in polar
coordinates defined by
\begin{gather*}
  \power^k \bsmat \cos t\\ v\sin t\esmat = \bsmat \cos kt \\ v \sin kt\esmat.
\end{gather*}
This definition yields real analytic maps $\Sph^n \to \Sph^n$ that generalize the
algebraically defined $k$-th powers of the unit spheres in the normed division algebras.

\begin{rem}
The degree of $\power^k$ is $k$ if $n$ is odd. If $n$ is even, the degree of
$\power^k$ is $1$ if $k$ is odd and $0$ if $k$ is even. This is classical.
For a generalization of the $k$-powers to cohomogeneity one
manifolds and a unified computation of the degree see \cite{puttmann}.
\end{rem}

Given a map $\rho : \Sph^{n-1} \to \Sph^{m-1}$,
view $\Sph^n\subset \R\times\R^n$ and $\Sph^m \subset \R\times\R^m$
as suspensions of $\Sph^{n-1}$ and $\Sph^{m-1}$ and let
\begin{gather*}
  \Sigma\rho : \Sph^n\to \Sph^m, \quad
  \bsmat x\\ v\esmat \mapsto \bsmat x \\ \abs{v} \rho(v/\abs{v})\esmat
\end{gather*}
denote the {\em suspension} of $\rho$.

\begin{lem}
\label{spherecommute}
We have the trivial identity
\begin{gather*}
  (\Sigma\rho)\circ \power^k\, = \,\power^k \circ \, \Sigma\rho.
\end{gather*}
\end{lem}

If $m = 2\ell$ and $k = 2j$ are even then the degree of the map $\power^k$
on the right hand side is zero.
Hence, $(\Sigma\rho)\circ \power^k$ is null-homotopic.
We describe an explicit null-homotopy in the following. Set
\begin{gather*}
  H_1: \Sph^n \times [0,1] \to \Sph^{2\ell}, \quad
  \bigl(\bsmat \cos t\\ v\sin t\esmat, s\bigr) \mapsto
  \begin{cases}
    \bsmat \cos 2jt\\ \rho(v)\sin 2jt\esmat, & \text{for $t \le \pi/2$,}\\
    \bsmat \cos 2jt\\ A(s)\cdot \rho(v) \sin 2jt\esmat, & \text{for $t \ge \pi/2$.}
  \end{cases}
\end{gather*}
Here, $A(s)\in \SO(2\ell)$ is a path from $\imat$ to $-\imat$, for example given by
$\ell$ copies of a standard $2\times 2$ rotation matrix. The homotopy $H_1$
deforms $(\Sigma\rho)\circ \power^{2j}$ to the map
\begin{gather*}
  \bsmat \cos t\\ v\sin t\esmat \mapsto  \bsmat\cos 2jt\\ \rho(v)\,\abs{\sin 2jt}\esmat.
\end{gather*}
Now set
\begin{gather*}
  H_2 : \Sph^n \times [0,1] \to \Sph^{2\ell}, \quad
  \bigl(\bsmat \cos t\\ v\sin t\esmat, s\bigr) \mapsto
  \bsmat x(s,t)\\  \rho(v) \sqrt{1-x(s,t)^2}\esmat
\end{gather*}
where $x(s,t) = s +(1-s)\cos 2jt$. The following is now immediate.

\begin{lem}
\label{spherenull}
The concatenation of the homotopies $H_1$ and $H_2$ deforms
$(\Sigma\rho)\circ \power^{2j}$ to the constant map to the north pole of
$\Sph^{2\ell}$.
\end{lem}

Similarly, for even $m=2l$ and odd $k= 2j+1$ one can explicitly deform
$(\Sigma\rho)\circ \power^{2j+1}$ to $\Sigma\rho$.

\smallskip

\subsection{Horizontal lifts}
We now explore Duran's specific form of the boundary map in
the exact homotopy sequence of a smooth fiber bundle
$F \cdots E \stackrel{\pi}\longrightarrow B$.
Assume that $E$ is equipped with an Ehresmann connection\footnote{An Ehresmann
connection is a complete horizontal distribution. In all our examples we fix a Riemannian
metric on the compact manifold $E$ such that $E\to B$ is a Riemannian submersion for
some Riemannian metric on $B$.
The horizontal space at a point $x\in E$ is then given by the orthogonal complement of
the tangent space to the fiber $F_x$. The Riemannian metric on $B$ is irrelevant in our
following constructions.}.
Suppose we are given a map $\alpha : \Sph^k \to B$.
Let $N = (1,0,\ldots,0)$ denote the north pole of $\Sph^k\subset \R^{k+1}$,
let $\Sph^{k-1}$ be the set of all vectors $v \in T_N \Sph^k$ with $\abs{v} = 1$,
and let $\gamma_v$ denote the geodesic of $\Sph^k$ with $\gamma_v(0) = N$
and $\dot \gamma_v(0) = v$. Fix a point $p \in E$.
Lift the curve $\alpha\circ\gamma_v$ in $B$ horizontally to a curve
$\widetilde{\alpha\circ\gamma}_v$ in $E$ with $\widetilde{\alpha\circ\gamma}_v(0) = p$.
Since $\gamma_v(\pi)$ is the south pole $S$ of $\Sph^k$ the end point
$\widetilde{\alpha\circ\gamma}_v(\pi)$ of the lifted curve is contained in the
fiber $F_{\vert \alpha(S)} = \pi^{-1}(\alpha(S))$ for any unit tangent vector $v$.
Set
\begin{gather*}
  \partial(\alpha): \Sph^{k-1} \to F_{\vert \alpha(S)} \subset E, \qquad
  v \mapsto \widetilde{\alpha\circ\gamma}_v(\pi).
\end{gather*}

\begin{lem}[see \cite{duran}]
The assignment $\alpha \mapsto \partial(\alpha)$ induces the boundary map
\begin{gather*}
  \pi_k(B) \to \pi_{k-1}(F)
\end{gather*}
in the exact homotopy sequence of the fibration $F \cdots E \to B$.
\end{lem}

\begin{proof}
This is one of the many equivalent topological constructions of the boundary map,
see e.\,g.~\cite{bredon}, page 452 for an appropriate reference.
The specific point here is that the lifting is performed horizontally.
\end{proof}

The Eckmann-Kervaire identity (see \cite{kervaire}) holds also for our specific way of lifting:
\begin{lem}
\label{susplift}
We have
\begin{gather*}
  \partial(\alpha\circ \Sigma\beta) = \partial(\alpha)\circ \beta.
\end{gather*}
\end{lem}

\begin{proof}
It is straightforward to verify that
$\alpha\circ\Sigma\beta\circ\gamma_v(t) = \alpha\circ\gamma_{\beta(v)}$.
\end{proof}

\smallskip

Now consider the case where a connected compact Lie group $G \subset \SO(m)$
acts transitively on a manifold $B$. Suppose that $\SO(m)$ is equipped with a biinvariant
Riemannian metric. This induces a biinvariant metric on the subgroup $G$.
Let $\imat$ denote the unit element in $G$, set $p = \pi(\imat)$,
and let $H = G_p$ denote the isotropy group at $p$. Then $H \cdots G \to B$
is a principal fiber bundle and we can apply the construction above.
Assume moreover that a subgroup $K \subset G$ also acts transitively on $B$.
The isotropy group $K_p$ is equal to $K\cap H$. We have a commutative diagram
\begin{gather*}
  \begin{CD}
    K\cap H @>>> H @>>> H / (K\cap H)\\
     @VVV    @VVV @|\\
    K @>>> G @>>> G/K\\
    @VV{\pi}V     @VV{\pi}V\\
    B @= B
  \end{CD}
\end{gather*}
where the equality $H/(K\cap H) = G/K$ holds since $H/(K\cap H)$ is included in
$G/K$ and both spaces have the same dimension. Note that the base manifold $B$
usually inherits different left invariant Riemannian metrics from $K$ and $G$
by Riemannian submersion but this is irrelevant for the lifting construction.

\begin{lem}
\label{doublehor}
If $\gamma : [0,T] \to B$ is a curve with $\gamma(0) = \pi(\imat)$ and $\tilde\gamma^K$,
$\tilde\gamma^G$ are the unique horizontal lifts of $\gamma$ with
$\tilde\gamma^K(0) = \tilde\gamma^G(0) = \imat$ then
\begin{gather*}
  \tilde\delta(t) = \tilde\gamma^G (t)^{-1} \cdot \tilde\gamma^K(t)
\end{gather*}
defines a curve in the fiber $\pi^{-1}(\pi({\imat})) = H$ and this curve is horizontal with respect to
the Riemannian submersion $H \to H / (K\cap H)$.
\end{lem}

\begin{proof}
Since $\tilde\gamma^K(t)$ and $\tilde\gamma^G(t)$ are lifts of $\gamma$, we have
$\gamma(t) = \tilde\gamma^K(t)\cdot p = \tilde\gamma^G(t)\cdot p$.
Hence $\delta(t) \in H$. Moreover,
\begin{gather*}
  \dot{\tilde\delta}(t) = - \tilde\gamma^G (t)^{-1}\cdot \dot{\tilde\gamma}^G (t)
  \cdot \tilde\gamma^G (t)^{-1} \cdot
  \tilde\gamma^K(t) + \tilde\gamma^G(t)^{-1} \cdot \dot{\tilde\gamma}^K(t)
\end{gather*}
and
\begin{gather*}
  \tilde\delta(t)^{-1} \cdot \dot{\tilde\delta}(t) =
   - \tilde\delta(t)^{-1} \cdot \tilde\gamma^G (t)^{-1} \cdot
   \dot{\tilde\gamma}^G (t) \cdot \tilde\delta(t)
   +  \tilde\gamma^K (t)^{-1} \cdot \dot{\tilde\gamma}^K (t).
\end{gather*}
Since $\tilde\gamma^K$ is horizontal, the second summand is perpendicular to
the Lie algebra $\ak\cap\ah$ of $K\cap H$. Since $\tilde\delta(t)\in H$ and $\tilde\gamma^G$
is horizontal, the first summand is perpendicular to the Lie algebra $\ah$ of $H$.
Hence, $\tilde\delta(t)^{-1} \cdot \dot{\tilde\delta}(t)$
is perpendicular to $\ah\cap \ak$ and $\tilde\delta$ is horizontal.
\end{proof}

\begin{cor}
For any map $\alpha : \Sph^k \to B$ the two maps $\partial_{G\to B}(\alpha)$
and $\partial_{K\to B}(\alpha)$ are homotopic within $\pi^{-1}(\alpha(S)) \approx H$
by the homotopy
\begin{gather*}
  (v,s) \mapsto \widetilde{\alpha\circ\gamma}_v^G(\pi) \cdot
  \widetilde{\alpha\circ\gamma}_v^G(s)^{-1} \cdot
  \widetilde{\alpha\circ\gamma}_v^K(s).
\end{gather*}
Here, $S = (-1,0,\ldots,0)$ denotes the south pole of $\Sph^k$.
\end{cor}

Assume now more specifically that the compact Lie group $G\subset \SO(n+1)$
acts transitively on $\Sph^n$. The isotropy group $G_N$ of the north pole $N$
is denoted by $H$. As above let $\gamma_v$ denote the geodesic of $\Sph^n$
with $\gamma_v(0) = N$ and $\dot \gamma_v(0) = v$, $\abs{v} = 1$.
Clearly $\gamma_v(t+\pi) = -\gamma_v(t)$ and $\dot\gamma_v(k\pi) = (-1)^k\cdot v$.
Note that the biinvariant metric on $G$ does not necessarily induce a constant curvature metric
on $\Sph^n$ via Riemannian submersion $H\cdots G \to \Sph^n$.
Hence, $\gamma_v$ should actually just be called a curve.
Let $\tilde\gamma_v$ be the unique horizontal lift of $\gamma_v$
with $\tilde\gamma_v(0) = \imat$.
Then $\dot{\tilde\gamma}_v(0) = (v,\ast)$ is a matrix in the Lie algebra $\ag$ of $G$
whose first column is $v$.

\begin{lem}
We have $\tilde\gamma_v(t+\pi) = \tilde\gamma_v(t) \cdot \tilde\gamma_v(\pi)$,
in particular, $\tilde\gamma_v(k\pi) = \tilde\gamma_v(\pi)^k$.
\end{lem}

\begin{proof}
Since $\tilde\gamma_v$ is horizontal,
$\tilde\gamma_v(t)^{-1}\cdot \dot{\tilde\gamma}_v(t)$
is perpendicular to the Lie algebra $\ah$ of $H$.
Now consider the curve
\begin{gather*}
  \sigma(t) = \tilde\gamma_v(t) \cdot \tilde\gamma_v(\pi).
\end{gather*}
Since
\begin{gather*}
  \sigma(t) = (\gamma_v(t),\ast) \cdot \bsmat -1 & \m 0 \\ \m 0 & \m \ast \esmat
  = (-\gamma_v(t),\ast) = (\gamma_v(t+\pi),\ast),
\end{gather*}
the curve $\sigma(t)$ is a lift of $\gamma_v(t+\pi)$ with $\sigma(0) = \gamma_v(\pi)$.
Moreover,
\begin{gather*}
  \sigma(t)^{-1} \cdot \dot \sigma(t) =  \tilde\gamma_v(\pi)^{-1} \cdot
   \bigl(\tilde\gamma_v(t)^{-1}\cdot \dot{\tilde\gamma}_v(t) \bigr)
      \cdot \tilde\gamma_v(\pi).
\end{gather*}
Since $\gamma_v(\pi) = \bmat -1\\ 0\emat$ and the isotropy group at $\bmat -1\\ 0\emat$
is also $H$, the lift $\tilde\gamma_v(\pi)$ normalizes $H$. Hence, $\sigma$ is horizontal.
By uniqueness, $\sigma(t) = \tilde\gamma_v(t+\pi)$.
\end{proof}

\begin{cor}
\label{powerid}
We have
\begin{gather*}
  \partial(\power^k) = \partial(\id_{\Sph^n})^k.
\end{gather*}
where $\power^k : \Sph^n \to \Sph^n$ denotes the $k$-th power of $\Sph^n$.
\end{cor}

\begin{lem}
\label{equiv}
The characteristic map $\partial(\id_{\Sph^n}) : \Sph^{n-1} \to G_S = H$
is equivariant with respect to the $H$ action on $\Sph^{n-1}$ and the adjoint action on $H$.
\end{lem}

\begin{proof}
The curves $\tilde\gamma_{hv}(t)$ and $h\tilde\gamma_v(t)h^{-1}$ are both horizontal,
the first one by definition, the second one because the isotropy group $G_{hp}$ is equal
to $hG_p h^{-1}$. Both curves start at the unit matrix $\imat$ and have initial velocity $hv$.
Hence they are identical. Evaluation at time $\pi$ yields the statement.
\end{proof}

\bigskip

\section{Identification of some characteristic maps}
\label{examples}
In this section we derive explicit formulas for the characteristic maps of the principal bundles
\begin{gather*}
  \SU(n)\cdots\SU(n+1)\to \Sph^{2n+1},\\
  \Syp(n)\cdots\Syp(n+1)\to\Sph^{4n+3},\\
  \SU(3)\cdots\Gtwo\to\Sph^6,\quad\text{ and }\quad
  \Gtwo\cdots\Spin(7)\to\Sph^7
\end{gather*}
using the horizontal lifting construction of section\,\ref{lifts}. We identify our maps
with maps previously given in the literature in other contexts. Finally, we prove
Theorem\,\ref{boundarymaps}. Note that the characteristic maps of the principal bundles
\begin{gather*}
  \SO(n+1) \to \Sph^{n},\quad
  \U(n+1) \to \Sph^{2n+1}, \quad
  \Syp(n+1) \to \Sph^{4n+3}
\end{gather*}
were already computed before 1950 by a different method, namely, by analyzing how the
bundle is glued from two local trivializations over the equatorial sphere (see \cite{steenrod}).
Both methods produce homotopic maps. Using horizontal lifts, however,
these maps come in a form that is much more suitable for our purposes.
In particular, they obey Corollary\,\ref{powerid}. Moreover, the special unitary case
was not considered in \cite{steenrod} and in this case we obtain the null-homotopy
$H_{\SU(3)\to\Sph^5}$ that is essential in our main construction.

\subsection{The characteristic map of the principal bundle $\SU(n+1)\to \Sph^{2n+1}$}
The characteristic map $\partial_{\SU(n+1)\to \Sph^{2n+1}}(\id)$ generates the homotopy group
\begin{gather*}
  \pi_{2n}(\SU(n))\approx\Z_{n!}.
\end{gather*}
(see \cite{bott} for the computation of this homotopy group). This follows easily by filling in only stable
homotopy groups into the exact homotopy sequence of the bundle.
In this subsection we identify $\partial_{\SU(n+1)\to \Sph^{2n+1}}(\id)$ with a map
$\phi$ given in \cite{unitary} and obtain the homotopy $H_{\SU(3)\to\Sph^5}$ of
Theorem\,\ref{finalthm}.

Let $\SU(n+1)$ denote the group of all complex $(n+1)\times (n+1)$-matrices $A$ with
$\bar A\tp A =\imat$. The group $\SU(n+1)$ acts transitively on the unit sphere $\Sph^{2n+1}$
in $\C^{n+1}$. We endow $\SU(n+1)$ with the up to scaling unique biinvariant Riemannian metric.
We consider the curve
\begin{gather*}
  \gamma(t) = \cos t \, \bsmat 1 \\ 0\esmat + \sin t \, \bsmat iy\\z\esmat
\end{gather*}
in $\Sph^{2n+1}$ with $y \in \R$, $z\in\C^n$, and $\abs{y}^2+\abs{z}^2 = 1$.
For $z \neq 0$ we set
\begin{multline*}
  \tilde \gamma(t) = \Biggl(
    \begin{pmatrix} 0 & 0\\ 0 & \, \imat -
      \smash[b]{\frac{z}{\abs{z}}\frac{\smash{\bar z\tp}}{\abs{z}}} \end{pmatrix}
    + \cos t \, \begin{pmatrix} 1 & 0\\ 0 & \; 
    \smash[b]{\frac{z}{\abs{z}} e^{iyt}
      \frac{\smash{\bar z\tp}}{\abs{z}}} \end{pmatrix} \\
    + \sin t \, \begin{pmatrix} iy & \; -\smash[t]{e^{iyt}\bar z\tp}\\ z & \;
    - \smash[b]{\frac{z}{\abs{z}} iy e^{iyt} 
    \frac{\smash{\bar z\tp}}{\abs{z}}} \end{pmatrix} \Biggr) \cdot
    \begin{pmatrix} 1 & 0\\ 0 & e^{-iyt/n}\,\imat \end{pmatrix}.
\end{multline*}
\begin{lem}
The curve $\tilde\gamma$ is the unique horizontal lift of $\gamma$ with respect to the
projection $\SU(n+1)\to\Sph^{2n+1}$ such that $\tilde \gamma(0) = \imat$.
In particular, the above formula for $\tilde\gamma$ extends analytically
to the case where $z=0$.
\end{lem}

\begin{proof}
It is evident that $\tilde\gamma(t)\cdot \bmat 1\\ 0\emat = \gamma(t)$. Hence,
$\tilde\gamma$ is a lift of $\gamma$. A direct computation shows that
\begin{gather*}
  \tilde\gamma(t)^{-1}\cdot \dot{\tilde\gamma}(t)
  = \begin{pmatrix} 1 & 0\\ 0 & e^{iyt/n}\,\imat \end{pmatrix} \cdot
   \begin{pmatrix} iy & -e^{iyt} \bar z\tp\\
   z e^{-iyt} & -\tfrac{iy}{n}\,\imat \end{pmatrix}
   \cdot \begin{pmatrix} 1 & 0\\ 0 & e^{-iyt/n}\,\imat \end{pmatrix}.
\end{gather*}
Hence, $\tilde\gamma(t)^{-1}\cdot \dot{\tilde\gamma}(t)$ is always perpendicular
to the Lie subalgebra $\asu(n)\subset\asu(n+1)$. Since the Riemannian
metric on $\SU(n+1)$ is left invariant this means that $\dot{\tilde\gamma}(t)$
is perpendicular to the $\SU(n)$-fiber at $\tilde\gamma(t)$.
\end{proof}

We recall a simple formula for a map that generates $\pi_{2n}(\SU(n))$ from \cite{unitary}.
This map
\begin{gather*}
  \phi : [0,\tfrac{2\pi}{n}] \times \Sph^{2n-1} \longrightarrow \SU(n)
\end{gather*}
is defined by
\begin{gather*}
  \phi(\tau,z) = A\cdot\diag(e^{i(n-1)\tau},e^{-i\tau},\ldots,e^{-i\tau})\cdot A^{-1}
\end{gather*}
where $A\in \SU(n)$ is any matrix whose first column is $z\in\Sph^{2n-1}\subset\C^n$.
The values of $\phi$ are independent of $z\in\Sph^{2n-1}$ for $\tau=0$
and $\tau=\tfrac{2\pi}{n}$. Because of this collapse at the endpoints of the interval
the map $\phi$ induces a map $\Sph^{2n}\to\SU(n)$. More precisely,
we now identify the round sphere $\Sph^{2n} \subset \R\times\C^n$ with the
cylinder $[0,\tfrac{2\pi}{n}] \times \Sph^{2n-1}$ that is collapsed at the endpoints
of the interval by the map
\begin{gather*}
  \Sph^{2n} \to \,([0,\tfrac{2\pi}{n}] \times \Sph^{2n-1}) / \x\sim\,,\quad
  \bmat i y\\z\emat \mapsto
  \bigl( \tfrac{\pi}{n}(y+1),\tfrac{z}{\abs{z}} \bigr).
\end{gather*}

\begin{thm}
We have
\begin{gather*}
  \partial_{\SU(n+1)\to \Sph^{2n+1}}(\id)\begin{pmatrix} iy \\ z \end{pmatrix}
  = \begin{pmatrix} -1 & 0 \\ \m 0 & e^{i\pi/n} \; \imat \end{pmatrix}
  \cdot \begin{pmatrix} 1 & 0\\ 0 & \phi\bigl( \frac{\pi}{n}(y+1),\frac{z}{\abs{z}} \bigr) \end{pmatrix}.
\end{gather*}
\end{thm}

\begin{proof}
The following computation shows that $\phi$ in fact only depends on the first column $z$ of
the matrix $A$:
\begin{gather*}
  \phi(\tau,z) = e^{-i\tau}\cdot A\cdot
    \left(\imat + \diag(e^{in\tau}-1,0,\ldots,0)\bigr)\cdot A^{-1}
  =  e^{-i\tau}\bigl(\imat + z(e^{in\tau}-1)\bar z\tp\right).
\end{gather*}
Straightforward evaluations of $\tilde\gamma(\pi)$ and
$\phi\bigl( \frac{\pi}{n}(y+1),\frac{z}{\abs{z}} \bigr)$ now yield the statement.
\end{proof}

We consider now the special case $n=2$. Recall the following homotopy from~\cite{unitary}:
\begin{multline*}
  H_{\SU(3)\to\Sph^5}(\tau,z,s)
    = \bmat z & -\bar w\\ w & \m \bar z\emat \bmat e^{i\tau} & 0 \\ 0 & e^{-i\tau}\emat
     \bmat \m \bar z & \m \bar w\\ -w & \m  z\emat \cdot \\
      \bmat z & -\bar w\\ w & \m \bar z\emat \bmat \cos s & -\sin s\\ \sin s & \m \cos s\emat
      \bmat e^{i\tau} & 0 \\ 0 & e^{-i\tau}\emat \bmat \m \cos s & \m \sin s \\ -\sin s & \m \cos s\emat
     \bmat \m \bar z & \m \bar w\\ -w & \m  z\emat.
\end{multline*}

\begin{thm}[see \cite{unitary}]
\label{suhom}
The homotopy $H_{\SU(3)\to\Sph^5}$ deforms $\phi^2$ and hence the map
\begin{gather*}
  \left(\bmat -1 & \m 0 & \m 0\\ \m 0 & -i & \m 0\\ \m 0 & \m 0 & -i \emat
  \cdot \partial_{\SU(n+1)\to \Sph^{2n+1}}(\id)\right)^2
\end{gather*}
(for $s=0$) to the constant map to the unit matrix (for $s = \frac{\pi}{2}$).
\end{thm}

\begin{rem}
The cases $n > 2$ were also treated in \cite{unitary}. Let $\eta: \Sph^{2n-1} \to \SU(n)$
be a map that generates $\pi_{2n-1}(\SU(n))$ and let $\eta_j$ denote its $j$-th column.
The $n$ maps $\phi\circ\Sigma\eta_j:\Sph^{2n}\to\SU(n)$ are all mutually homotopic
(by homotopies analogous to the above) and represent the $(n-1)!$-th power of the
generator $\phi$ of $\pi_{2n}(\SU(n))$. The maps $\phi\circ\Sigma\eta_j$ commute
mutually and their value-by-value product is the constant map to the identity matrix.
\end{rem}

\smallskip

\subsection{The characteristic map of the principal bundle $\Syp(n+1)\to \Sph^{4n+3}$}
The characteristic map of the principal bundle $\Syp(n+1)\to \Sph^{4n+3}$ generates
the homotopy group
\begin{gather*}
  \pi_{4n+2}(\Syp(n))\approx
  \begin{cases} \Z_{(2n+1)!}, & \text{if $n$ is even,}\\ \Z_{2(2n+1)!}, & \text{if $n$ is odd} \end{cases}
\end{gather*}
(see \cite{kervpont} for the computation of this homotopy group). This follows easily by filling in
only stable homotopy groups into the exact homotopy sequence of the bundle. In this subsection we
compute the characteristic map using the horizontal lifting construction of section\,\ref{lifts}.

Let $\Syp(n+1)$ denote the group of all quaternionic $(n+1)\times (n+1)$-matrices $A$
with $\bar A\tp A =\imat$.
The group $\Syp(n+1)$ acts transitively on the unit sphere $\Sph^{4n+3}$ in~$\H^{n+1}$.
We endow $\Syp(n+1)$ with the up to scaling unique biinvariant Riemannian metric.
We consider the curves
\begin{gather*}
  \gamma(t) = \cos t \, \bmat 1 \\ 0\emat + \sin t \, \bmat p\\u\emat
\end{gather*}
in $\Sph^{4n+3}$. Here, $p \in \Imag\H$ is purely
imaginary and $u\in\H^n$ is a vector such that $\abs{p}^2+\abs{u}^2 = 1$.
In other words the vector $\bmat p\\u\emat$ is contained in the unit sphere
$\Sph^{4n+2}$. For $u\neq 0$ we set
\begin{gather*}
  \tilde \gamma(t) =
    \Bsmat 0 & 0\\ 0 & \, \imat -
      \smash[b]{\tfrac{u}{\abs{u}}\tfrac{\smash{\bar u\tp}}{\abs{u}}} \Esmat
    + \cos t \, \bbsmat 1 & 0\\ 0 & \, 
    \smash[b]{\tfrac{u}{\abs{u}} e^{tp} \tfrac{\smash{\bar u\tp}}{\abs{u}}} \eesmat
    + \sin t \, \bbsmat p & \; -\smash{e^{tp} \bar u\tp}\\ u & \;
    - \smash[b]{\tfrac{u}{\abs{u}} p e^{tp} \tfrac{\smash{\bar u\tp}}{\abs{u}}} \eesmat,
\end{gather*}
where $e^{p} = \cos \abs{p} + \tfrac{p}{\abs{p}} \sin \abs{p}$ denotes the
exponential map of $\Sph^3\subset\H$ from $1$.

\begin{lem}[see \cite{duran} for $n=1$]
The curve $\tilde\gamma$ is the unique horizontal lift of $\gamma$ with respect to the
projection $\Syp(n+1)\to\Sph^{4n+3}$ such that $\tilde \gamma(0) = \imat$.
In particular, the above formula for $\tilde\gamma$ extends in an analytic way
to the case where $u=0$.
\end{lem}

\begin{proof}
It is evident that $\tilde\gamma(t)\cdot \bmat 1\\ 0\emat = \gamma(t)$. Hence,
$\tilde\gamma$ is a lift of $\gamma$. A direct computation shows that
\begin{gather*}
  \tilde\gamma(t)^{-1}\cdot \dot{\tilde\gamma}(t)
  = \Bsmat p & -\smash{e^{tp} \bar u\tp}\\ u e^{\!-tp} & 0 \Esmat.
\end{gather*}
Hence, $\tilde\gamma(t)^{-1}\cdot \dot{\tilde\gamma}(t)$ is always perpendicular
to the Lie subalgebra $\asyp(n-1)\subset\asyp(n)$. Since the Riemannian
metric on $\Syp(n+1)$ is left invariant this means that $\dot{\tilde\gamma}(t)$
is perpendicular to the $\Syp(n-1)$-fiber at $\tilde\gamma(t)$.
\end{proof}

Note that the fiber of the bundle $\Syp(n+1) \to \Sph^{4n+3}$ over the south pole
$\bmat -1\\ 0\emat$ is $\bmat -1  & \m 0 \\ \m 0 & \m \Syp(n)\emat$. This fiber
can be canonically identified with the $\Syp(n)$ fiber over the north pole using
left multiplication by $\bmat -1 & \m 0\\ \m 0 & \m \imat\emat\in \Syp(n)$.

\begin{cor}[see \cite{duran} for $n=1$]
The characteristic map $\partial_{\Syp(n+1)\to \Sph^{4n+3}}(\id) : \Sph^{4n+2}\to \Syp(n)$
is given by
\begin{gather*}
  \partial_{\Syp(n+1)\to \Sph^{4n+3}}(\id)\bmat p\\ u\emat
   = \imat
    -  \smash[b]{\tfrac{u}{\abs{u}} (1+e^{\pi p}) \tfrac{\smash{\bar u\tp}}{\abs{u}}}
\end{gather*}
and hence, for $n = 1$, by
\begin{gather*}
  \partial_{\Syp(2)\to \Sph^{7}}(\id)\bmat p\\ u\emat
   =   -  \smash[b]{\tfrac{u}{\abs{u}} e^{\pi p} \tfrac{\smash{\bar u}}{\abs{u}}}.
\end{gather*}
\end{cor}

\smallskip

\subsection{The characteristic map of the principal bundle $\Gtwo\to \Sph^6$}
In this subsection we identify the characteristic map of the principal bundle $\Gtwo\to \Sph^6$
with the embedding $\eta:\Sph^5\to\SU(3)$ given in \cite{unitary} (and previously in slightly
different form in \cite{chaves} and \cite{bryant}).

Let $\Ca = \C\times \C^3$ denote the octonions with the product introduced in section\,\ref{oct}.
Octonionic conjugation is given by $\bmat z_0\\ z\emat \mapsto \bmat \bar z_0\\ -z\emat$.
Let $\Imag\Ca = i\R\times\C^3$ denote the imaginary octonions.
The compact Lie group $\Gtwo$ is the automorphism group of the octonions. It is a connected
subgroup of $\SO(\Imag\Ca)$ and acts transitively on the sphere $\Sph^6\subset \Imag\Ca$.
Let $\gamma: \R \to \Sph^6$ be the unit speed geodesic defined by
\begin{gather*}
  \gamma(t) = \bmat i\cos t \\ \sin t\\ 0\\ 0\emat.
\end{gather*}
Set
\begin{gather*}
  e_1(t) = \gamma(t), \quad
  e_2(t) = \bmat -i\sin t\\ \cos t\\ 0\\ 0\emat, \quad
  e_3(t) = \bmat 0\\ i\\ 0\\ 0\emat, \\
  e_4(2t) = \bmat 0\\ 0\\ \cos t\\ i\sin t\emat, \;
  e_5(2t) = \bmat 0\\ 0\\ i\cos t\\ \sin t\emat, \;
  e_6(2t) = \bmat 0\\ 0\\ -i\sin t\\ \cos t\emat, \;
  e_7(2t) = \bmat 0\\ 0\\ -\sin t\\ i\cos t\emat.
\end{gather*}
and define a curve $\tilde\gamma(t)$ in $\SO(\Imag\Ca)$ by mapping $e_k(0)$ to
$e_k(t)$ for each $k$.

\begin{lem}
With respect to the fibration $\Gtwo\to \Sph^6$, $A\mapsto A\bigl(e_1(0)\bigr)$,
the curve $\tilde\gamma$ is the unique horizontal lift of $\gamma$ with $\tilde\gamma(0) =\imat$.
\end{lem}

\begin{proof}
It is easily verified using the octonionic multiplication from section\,\ref{oct} that
$e_1(t) \cdot e_2(t) = e_3(t)$, $e_5(t) = e_1(t)\cdot e_4(t)$, $e_6(t) = e_2(t)\cdot e_4(t)$,
and $e_7(t) = -e_3(t)\cdot e_4(t)$ for all times $t$. Hence, the curve $\tilde\gamma$
is contained in $\Gtwo$. Clearly, $\tilde\gamma$ is a lift of $\gamma$. It is also
easily verified that $\tilde\gamma$ is a one parameter subgroup of $\SO(\Imag\Ca)$.
Hence, $\tilde\gamma$ is a geodesic in $\Gtwo$ with respect to the up to scaling
unique biinvariant metric. It is a standard fact that in order to verify whether a geodesic
is horizontal, it suffices to verify horizontality at one point of time.
The fiber over $\gamma(0) = \imat$ is the natural embedding of $\SU(3)$ into $\SO(\Imag\Ca)$
as the automorphism group of the complex cross product of section\,\ref{oct}.
It is straightforward to verify that $\dot{\tilde\gamma}(0)$ is perpendicular to this $\SU(3)$ fiber.
\end{proof}

We now recall the embedding $\eta:\Sph^5\to \SU(3)$ from \cite{unitary} (see also the previous papers
\cite{chaves} and \cite{bryant}):
Let $\Sph^5$ be the unit sphere in $\C^3$. Set
\begin{gather*} 
  \eta(z) = zz\tp + \bmat 0 & -\bar z_3 & \bar z_2\\ \bar z_3 & 0 & -\bar z_1\\ -z_2 & \bar z_1 & 0\emat.
\end{gather*}
Note that $\eta(Az) = A\eta(z)A\tp$ for $A \in \SU(3)$ and that $\eta$ generates $\pi_5(\SU(3))$.
Let $\theta \in \SO(\Imag\Ca)$ denote complex (not octonionic) conjugation on
$\Imag\Ca = i\R \times\C^3$. The fiber of the bundle $\Gtwo\to\Sph^6$ over $\bmat -i\\ 0\emat$
is evidently $\SU(3)\cdot \theta$ since $\theta\bmat i\\ 0\emat = \bmat -i\\ 0\emat$.

\begin{thm}
\label{eta}
We have
\begin{gather*}
  \partial_{\Gtwo\to \Sph^6}(\id)(z) = \eta(-iz)\cdot \theta.
\end{gather*}
\end{thm}

\begin{proof}
Straightforward evaluation shows $\tilde\gamma(\pi) = \eta\bmat -i\\ 0\emat\cdot\theta$.
By Lemma\,\ref{equiv} the characteristic map is equivariant with respect to the transitive action
of $\SU(3)$ on $\Sph^5$ and to the action of $\SU(3)$ on $\SU(3)\cdot \theta$ by conjugation.
We have
\begin{multline*}
  \partial_{\Gtwo\to \Sph^6}(\id)\left(A\cdot\bmat -i\\ 0\emat\right)
  = A \cdot \eta\bmat -i\\ 0\emat \cdot \theta\cdot A^{-1} \\
  = A \cdot \eta\bmat -i\\ 0\emat \cdot A\tp\cdot \theta
  = \eta \left(-i \,A\cdot\bmat -i\\ 0\emat\right)\cdot \theta
\end{multline*}
since $\eta(Az) = A\eta(z)A\tp$.
\end{proof}

\begin{rem}
Theorem\,\ref{eta} in particular implies the known fact that $\partial_{\Gtwo\to \Sph^6}(\id)$
generates $\pi_5(\SU(3)) \approx\Z$. This can be used to give an elementary proof of the
known fact that $\pi_5(\Gtwo)$ is trivial (see \cite{mimura}) just by inspecting the
relevant part of the exact homotopy sequence of the bundle $\Gtwo\to \Sph^6$.
\end{rem}

\begin{rem}
Observe that $\partial_{\Gtwo\to \Sph^6}(\id)^4 = \imat$, while $\partial_{\Gtwo\to \Sph^6}(\id)$
generates an infinite cyclic homotopy group. This is reflected in the fact that when transfering
the boundary map from the fiber over the south pole to the fiber over the north pole the
equivariance changes from conjugation to twisted conjugation (twisted by an outer automorphism
of $\SU(3)$).
\end{rem}

\begin{rem}
It is known (see, e.g.,~\cite{kollross}) that the action of $\SU(3)\times\SU(3)$ on $\Gtwo$
by left and right translations is of cohomogeneity one. The geodesic $\tilde\gamma$
is a normal geodesic for this action, i.e., perpendicular to all orbits.
\end{rem}

\begin{rem}
The embedding $\eta$, i.e., essentially the characteristic map
$\partial_{\Gtwo\to \Sph^6}(\id)$, defines a calibrated submanifold of $\SU(3)$.
This result is due to R.~Bryant \cite{bryant} who classified all codimension\,3
cycles of compact Lie groups that are calibrated by the Hodge dual of the
fundamental $3$-form $(X,Y,Z) \mapsto \langle X,[Y,Z]\rangle$.
\end{rem}

\subsection{The characteristic map of the principal bundle $\Spin(7)\to \Sph^7$}
Chaves and Rigas obtained in \cite{rigas} an embedding $\psi:\Sph^6\to \Gtwo$
that genrates $\pi_6(\Gtwo)\approx\Z^3$. This embedding parametrizes the adjoint
orbit of $\Gtwo$ through one of the elements in the center of the subgroup
$\SU(3)\subset\Gtwo$ and is hence minimal. In this subsection we obtain the
embedding $\psi$ by the horizontal lifting construction. This approach is essential
for the subsequent proof of Theorem\,\ref{boundarymaps}.

Let $\Ca$ denote the normed division algebra of the octonions.
We use the triality model of $\Spin(8)$:
\begin{gather*}
  \Spin(8) = \{ (A,B,C) \in\SO(\Ca)^3 \cond A(x\cdot y) = B(x)\cdot C(y)
  \text{ for all $x,y\in\Ca$}\}
\end{gather*}
and consider the natural homomorphism $\Spin(8)\to\SO(8)$ given by the
projection to the first factor. From this homomorphism we get the standard
transitive action of $\Spin(8)$ on the unit octonions $\Sph^7$.
The condition $A(1) = 1$ defines a subgroup $\Spin(7)\subset\Spin(8)$,
which acts transitively on the unit octonions by
\begin{gather*}
  \Spin(7) \times \Sph^7 \to \Sph^7, \quad
  (A,B,C) \cdot x  = B(x).
\end{gather*}
The isotropy group of $1\in\Ca$ with respect to this action is the automorphism
group of the octonions
\begin{gather*}
  \Gtwo = \{(A,A,A) \in \Spin(7) \}.
\end{gather*}
From the action of $\Spin(7)$ we therefore get the $\Gtwo$-principal bundle $\Spin(7)\to\Sph^7$ in the usual way: $(A,B,C) \mapsto (A,B,C)\cdot 1 = B(1)$.

\smallskip

Let $\gamma$ be a geodesic in the unit octonions with $\gamma(0) = 1$ and
$\gamma(\pi) = -1$. In order to give a formula for the horizontal lift of $\gamma$
to $\Spin(7)$ we use the standard notation for the left multiplication, right
multiplication, and conjugation on $\Ca$ with a fixed octonion $a$:
\begin{gather*}
  L_{a}(x) = ax,\quad
  R_{a} (x) = xa,\quad
  C_{a}(x) = axa^{-1}.
\end{gather*}
Since the octonions form a normed division algebra the transformations
$L_{a}$, $R_{a}$, and $C_{a}$ are contained in $\SO(\Ca)$.

\begin{lem}
\label{spinlift}
The horizontal lift of $\gamma$ through the unit element of $\Spin(7)$ is
given~by
\begin{gather*}
  \tilde \gamma(3t) = \bigl(C_{\gamma(t)}, L_{\gamma(t)}\circ R_{\gamma(t)^2},
  L_{\bar\gamma(t)^2} \circ R_{\bar\gamma(t)} \bigr)\,.
\end{gather*}
Here, $\bar\gamma(t)$ is identical to $\gamma(t)^{-1}$ since $\gamma(t)$
is a unit octonion.
\end{lem}

\begin{proof}
It follows from the Moufang identities
\begin{gather*}
  L_{aba} = L_{a}\circ L_{b}\circ L_{a},\quad
  R_{aba} = R_{a}\circ R_{b}\circ R_{a}
\end{gather*}
that the expression for $\tilde\gamma(3t)$ above is contained in $\Spin(8)$:
\begin{gather*}
  a(x y)\bar a = a \bigl( x (a(\bar a y))\bigr)\bar a
  = \bigl((axa)(\bar a y)\bigr)\bar a
  = \bigl(\bigl((a x a^2)\bar a \bigr)(\bar a y)\bigr)\bar a
  = (a x a^2)(\bar a^2 y\bar a).
\end{gather*}
Moreover, $\tilde\gamma(3t)$ is contained in $\Spin(7)$ since $C_{a}(1) = 1$.
The formula for $\tilde\gamma$ apparently defines a one parameter subgroup
of $\Spin(7)$ and hence a geodesic. This geodesic projects to the geodesic
$\gamma$ of $\Sph^7$ since
\begin{gather*}
  (C_{a}, L_{a}\circ R_{a^2}, L_{\bar a^2} \circ R_{\bar a} )
   \cdot 1 = a^3.
\end{gather*}
It remains to show that $\tilde\gamma$ is horizontal. Since we already know that
$\tilde\gamma$ is a geodesic, it suffices to verify that $\Dot{\tilde\gamma}(0)$
is perpendicular to the $\Gtwo$-fiber over $1\in\Sph^7$. For this purpose we
note that the inner product of two endomorphisms
$Y,Z\in\aso(\Ca)\times\aso(\Ca)\times\aso(\Ca)$ is given by $-\tr (Y\circ Z)$.
We have $\dot\gamma(0) = p\in\Sph^6\subset\Imag\Ca$. Hence,
\begin{gather*}
  \Dot{\Tilde\gamma}(0) = (L_p - R_p, L_p+2R_p, -2L_p -R_p)
  \in \aso(\Ca)\times\aso(\Ca)\times\aso(\Ca).
\end{gather*}
The Lie algebra of $\Gtwo$ consists of endomorphisms of the form
$(X,X,X)\in\aso(\Ca)\times\aso(\Ca)\times\aso(\Ca)$. It is now obvious that
$\tr(\Dot{\Tilde\gamma}(0)\circ (X,X,X))$ vanishes, since the sum of the three
components of $\Dot{\Tilde\gamma}(0)$ vanishes.
Therefore, $\tilde\gamma$ is horizontal.
\end{proof}

Let $\Sph^6\subset\Imag(\Ca)$ denote the imaginary unit octonions.
For any geodesic $\gamma_v(t)$ with $v\in\Sph^6$
we have $\gamma_v(\pi) = -1$. Thus, $\tilde\gamma_v(\pi)$ is contained in the
$\Gtwo$-fiber $\Spin(7)_{-1}\approx \Gtwo$. Since $(\imat,-\imat,-\imat)\in\Spin(7)$
left translates the north pole $1\in\Sph^7$ to the south pole $-1$ we have
$\Spin(7)_{-1} = (\imat,-\imat,-\imat)\cdot\Gtwo$.
\begin{cor}
\label{charspin}
The characteristic map $\partial_{\Spin(7)\to\Sph^7}(\id): \Sph^6 \to \Spin(7)_{-1}$
is given by
\begin{align*}
  v \mapsto & \bigl(C_{\gamma_v(\pi/3)},
  L_{\gamma_v(\pi/3)} \circ R_{\gamma_v(2\pi/3)},
  L_{\gamma_v(-2\pi/3)} \circ R_{\gamma_v(-\pi/3)}\bigr)\\
  =& (C_{\gamma_v(-2\pi/3)}, -C_{\gamma_v(-2\pi/3)},
  -C_{\gamma_v(-2\pi/3)}),
\end{align*}
in particular, $\bigl(\partial_{\Spin(7)\to\Sph^7}(\id)\bigr)^3 \equiv (\imat,-\imat,-\imat)$
and $\bigl(\partial_{\Spin(7)\to\Sph^7}(\id)\bigr)^6 \equiv (\imat,\imat,\imat)$.
\end{cor}

\begin{lem}
The characteristic map $\partial_{\Spin(7)\to\Sph^7}(\id): \Sph^6 \to \Spin(7)_{-1}$
generates the homotopy group $\pi_6(\Gtwo)\approx\Z_3$.
\end{lem}
\begin{proof}
This follows immediately from the exact homotopy sequence of the principal bundle
$\Spin(7)\to\Sph^7$:
\begin{gather*}
   \Z \approx \pi_{7}(\Sph^{7})
   \, \longrightarrow\, \pi_{6}(\Gtwo)
      \,\longrightarrow\, \pi_{6}(\Spin(7)) \approx 0.\qedhere
\end{gather*}
\end{proof}

Chaves and Rigas \cite{rigas} obtained the map
\begin{gather*}
  \psi: \Sph^6 \to \Gtwo, \quad
  p \mapsto (C_{\exp(-2\pi p/3)}, C_{\exp(-2\pi p/3)}, C_{\exp(-2\pi p/3)})
\end{gather*}
which parametrizes the adjoint orbit of $\Gtwo$ through one of the elements in the center
of $\SU(3)\subset\Gtwo$ and thus represents the homotopy group $\pi_6(\Gtwo)\approx\Z_3$
in the most geometric way possible: $\psi^{3j}$ is the constant map to the identity matrix
in $\Gtwo$, while the embeddings $\psi^{3j+1}$ and $\psi^{3j+2}$ represent the two
nontrivial classes in $\pi_6(\Gtwo)$.

\begin{thm}
We have
\begin{gather*}
  \partial_{\Spin(7)\to\Sph^7}(\id) = (\imat,-\imat,-\imat) \cdot \psi.
\end{gather*}
\end{thm}

With this identity we can, of course, also recover from Corollary\,\ref{charspin} that
$\psi^3$ is the constant map to the identity matrix. Moreover,
it follows from Lemma\,\ref{equiv} that $\partial_{\Spin(7)\to\Sph^7}(\id)$ is
$\Gtwo$-equivariant and hence it is easy to recover that the map $\psi$
indeed parametrizes an adjoint orbit of $\Gtwo$ through one of the elements
in the center of $\SU(3)\subset \Gtwo$.

\smallskip

\subsection{Proof of Theorem\,\ref{boundarymaps}}
\label{geomidentity}
Identify $\Sph^7$ with the unit sphere in the octonions $\Ca$ and let $V_{7,2}$
denote the Stiefel manifold of orthonormal $2$-frames in $\R^7 \approx \Imag\Ca$.
Let $e_1, e_2,\ldots$ denote an orthonormal basis of $\R^7 \approx \Imag\Ca$.
Define the map
\begin{gather*}
  \kappa : \Sph^7 \to V_{7,2}, \quad
  a \mapsto (\bar a e_1 a, \bar a e_2 a).
\end{gather*}
Identify $\Syp(2)$ with the subgroup of matrices
$(A,B,C)$ in $\Spin(7)\subset \SO(\Ca)\times\SO(\Ca)\times\SO(\Ca)$, such that,
$A\cdot e_1 = e_1$ and $A\cdot e_2 = e_2$, i.e., with $\Spin(5)$.

Let $\gamma_v$ denote the geodesic of $\Sph^7$
with $\gamma_v(0) = N$ and $\dot \gamma_v(0) = v$, $\abs{v} = 1$.
Apply Lemma\,\ref{doublehor} to the commutative diagram
\begin{gather*}
  \begin{CD}
    \Sph^3 @>>> \Gtwo @>>>  \Gtwo/\Sph^3 \\
     @VVV    @VVV @|\\
    \Syp(2) @>>> \Spin(7) @>>> \Spin(7)/\Syp(2) @= V_{7,2}\\
    @VV{\pi}V     @VV{\pi}V\\
    \Sph^7 @= \Sph^7
  \end{CD}
\end{gather*}
This shows that the curves
\begin{gather*}
  \tilde\delta_v(t) = \tilde\gamma_v^{\Spin(7)}(t)^{-1}\cdot \tilde\gamma_v^{\Syp(2)}(t)
\end{gather*}
in $\Gtwo$ are horizontal with respect to the fibration $\Gtwo\to V_{7,2}$.
The curves $\tilde\delta_v$ satisfy $\tilde\delta_v(0) = \imat$ and
$\tilde\delta_v(6\pi) = \tilde\gamma_v^{\Syp(2)}(6\pi)$ since
$\tilde\gamma_v^{\Spin(7)}(6\pi) = \imat$ by Lemma\,\ref{spinlift}.
The two column vectors of the projected curve $\delta(t)$
in $V_{7,2}$ are given by $\tilde\delta_v(t) \cdot e_1$ and $\tilde\delta_v(t) \cdot e_2$.
Here, $\tilde\delta_v(t) \in\Gtwo$ has three identical entries in
$\SO(\Imag \Ca) \times \SO(\Imag\Ca)\times \SO(\Imag\Ca)$.
Trivially, all of these act identically on $e_1$ and $e_2$. In order to evalute the
result it is suitable to choose the first component since this is adepted to the
embedding of $\Syp(2)$ into $\Spin(7)$. We obtain
\begin{gather*}
  \tilde\delta_v(3t) \cdot e_1
  = \tilde\gamma_v^{\Spin(7)}(3t)^{-1}\cdot \tilde\gamma_v^{\Syp(2)}(3t)\cdot e_1
  = \tilde\gamma_v^{\Spin(7)}(3t)^{-1}\cdot e_1
  = \overline{\gamma_v(t)}\cdot e_1\cdot \gamma_v(t)
\end{gather*}
and hence
$\tilde\delta_v(6t) \cdot e_1 = \overline{\gamma_v(2t)}\cdot e_1\cdot \gamma_v(2t)$.
Similarly,
$\tilde\delta_v(6t) \cdot e_2 = \overline{\gamma_v(2t)}\cdot e_2\cdot \gamma_v(2t)$.

\begin{rem}
Toda, Saito, and Yokota proved that the map $\kappa$ generates
$\pi_7(V_{7,2}) \approx \Z_4$. Our construction above yields a
direct argument for this fact:
The map $\kappa$ projects to a map $\Sph^7\to \Sph^6$ that generates $\pi_7(\Sph^6)$
since it can be deformed to the fourth suspension of the Hopf fibration,
see subsection\,\ref{kappatoh}.
The claim follows from the relevant part of the exact homotopy sequence of the fibration
$\Sph^5 \cdots V_{7,2} \to \Sph^6$ (see the introduction).
\end{rem}

\bigskip

\section{The precise relation between the commutator and the Duran map}
\label{Scommutator}
In this section we perform Construction\,\ref{homeq}.
The domain of definition of the commutator of unit quaternions is $\Sph^3\times\Sph^3$.
The commutator factors through the smash product $\Sph^3\wedge\Sph^3$.
It is elementary that the smash product $\Sph^3\wedge\Sph^3$ is homeo\-morphic to $\Sph^6$
since $\Sph^3\wedge\Sph^3$ is the one-point compactification of $\R^3\times\R^3$.
For our purposes, however, this homeomorphism is not appropriate, since it does
not have any direct relations to the characteristic map $\partial_{\Syp(2)\to\Sph^7}(\id)$.
Instead, we extract from Duran's explicit formula for $\partial_{\Syp(2)\to\Sph^7}(\id)$
a suitable homotopy equivalence between $\Sph^3\wedge\Sph^3$ and $\Sph^6$.

\smallskip

As usual, we define the one point union
\begin{gather*}
  \Sph^3 \vee\Sph^3 = \{1\}\times\Sph^3\cup \Sph^3\times \{1\}
\end{gather*}
and the smash product by
\begin{gather*}
  \Sph^3\wedge\Sph^3 =
    (\Sph^3\times\Sph^3) / (\Sph^3 \vee\Sph^3),
\end{gather*}
The value of the commutator $[a,b]$ is always $1$ if $a = 1$ or $b= 1$
thus the commutator $\comm$ factors through the smash product. Set
\begin{gather*}
  \iota: \Sph^6 \to \Sph^3\wedge\Sph^3, \quad
  \bsmat p\\ u\esmat \mapsto (u/\abs{u}, -e^{\pi p}).
\end{gather*}
Here, $\Sph^6$ denotes the unit sphere in $\Imag \H\times\H$ and
$e^{p}$ denotes the exponential map of $\Sph^3\subset \H$ at the point $1$
(note that $T_1\Sph^3 = \Imag \H$).

\begin{lem}
The map $\iota : \Sph^6\to \Sph^3\wedge\Sph^3$ is continuous.
\end{lem}

\begin{proof}
The map $\iota$ is well-defined since the first $\Sph^3$-factor collapses
if $u=0$ and hence $\abs{p}=1$. The smash product $\Sph^3\wedge\Sph^3$
inherits a canonical distance function $\bar d$ from the standard metric $d$ on
$\Sph^3\times\Sph^3$:
\begin{gather*}
  \bar d([x],[y]) =
    \min\bigl\{d(x,y), d(x,\Sph^3\vee\Sph^3) + d(y,\Sph^3\vee\Sph^3)\bigr\}.
\end{gather*}
The distance function $\bar d$ induces the quotient topology on
$\Sph^3\wedge\Sph^3$. Since
\begin{gather*}
  d(\iota\bmat p\\u\emat,\Sph^3\vee\Sph^3)\to 0
\end{gather*}
if $u\to 0$ the map $\iota$ is continuous.
\end{proof}

It is easy to see that $\iota$ is surjective: If $a,b \in \Sph^3$ with
$a \neq 1$ and $b \neq 1$ then the equation
\begin{gather*}
  \iota\bmat p\\u\emat = [(u/\abs{u}, -e^{\pi p})] = [(a,b)]
\end{gather*}
has precisely one solution.
If $a = 1$ or $b = 1$ then the solutions of the latter equation are precisely
of the form $\bmat p\\u\emat$ with $u\in\R$, $u\ge 0$. This is a three dimensional
disk $D^3$ in $\Sph^6$. The map $\iota$ hence factors  as follows:
\begin{gather*}
  \Sph^6 \longrightarrow \Sph^6/D^3 \stackrel{\lambda}\longrightarrow \Sph^3\wedge\Sph^3.
\end{gather*}
The second map $\lambda$ is continuous and bijective and therefore a homeomorphism.
Its inverse is the map
\begin{gather*}
  \lambda^{-1}: (a,b) \mapsto \bmat p\\ a \sqrt{\smash[b]{1-\abs{p}^2}}\emat, \quad
  \text{where $p$ is defined by $b = -e^{\pi p}$.}
\end{gather*}
The projection map $\Sph^6 \to \Sph^6/D^3$ is a homotopy equivalence by standard constructions.
Since this essential here, we present the details:
Let $f:\Sph^6\times [0,1]\to\Sph^6$ be a homotopy with $f_0 = \id_{\Sph^6}$,
$f_s(D^3)\subset D^3$ and such that  $f_1(D^3)$ consists only of one point.
An explicit formula for such a homotopy is, for example, given by
\begin{gather*}
   f_s\bmat p\\ u\emat = \begin{cases}
   \bmat 0\\ \Real u + s(1+\Real u)\emat +
     \sqrt{\tfrac{1-(\Real u + s(1+\Real u))^2}{1-(\Real u)^2}}
   \bmat p \\ \Imag u \emat & \text{for $\Real u \le \frac{1-s}{1+s}$,}\\
   \bmat 0 \\ 1\emat & \text{for $\Real u \ge \tfrac{1-s}{1+s}$}.
  \end{cases}
\end{gather*} 
We have a commutative diagram
\begin{gather*}
  \begin{CD}
    \Sph^6 @>{f_t}>> \Sph^6\\
     @VVV    @VVV \\
    \Sph^6/D^3 @>>> \Sph^6/D^3 \end{CD}
\end{gather*}
with $f_0 = \id$ and such that for $s=1$ the induced map $\Sph^6/D^3\to \Sph^6/D^3$
lifts to a map $\hat f_1:\Sph^6/D^3\to \Sph^6$. We have seen:
\begin{lem}
The projection map $\Sph^6 \to \Sph^6/D^3$ is a homotopy equivalence with
homotopy inverse $\hat f_1$. Hence, the map $\iota : \Sph^6\to \Sph^3\wedge \Sph^3$
is a homotopy equivalence with homotopy inverse $\mu := \hat f_1\circ \lambda^{-1}$.
\end{lem}

\begin{cor}
\label{homone}
The homotopy
\begin{gather*}
  \Sph^3\wedge \Sph^3 \stackrel{\lambda^{-1}}\longrightarrow \Sph^6/D^3 \stackrel{f_s}
  \longrightarrow \Sph^6/D^3 \stackrel{\lambda}\longrightarrow \Sph^3\wedge\Sph^3
  \stackrel{[\cdot,\cdot]}\longrightarrow \Sph^3
\end{gather*}
is a homotopy between the commutator $\comm$ and the composition $\comm\circ\iota\circ\mu$.
\end{cor}

Now consider the two homotopies $\Sph^6 \times [0,1] \to \Sph^3$,
\begin{gather*}
  (\bmat p\\u\emat,s) \mapsto
  \tfrac{u}{\abs{u}} e^{\pi p} \tfrac{\bar u}{\abs{u}} e^{-(1-s)\pi p+s\pi\qi} \quad\text{ and }\quad
  (\bmat p\\u\emat,s) \mapsto
  \tfrac{u}{\abs{u}} e^{12 \pi p} \tfrac{\bar u}{\abs{u}} e^{-(1-s) 12 \pi p}.
\end{gather*}
The first homotopy was given in \cite{duranrigas}, the other is a simple modification.

\begin{lem}
\label{homtwo}
The first homotopy deforms $\comm\circ\iota$ to the characteristic map
$\partial_{\Syp(2)\to\Sph^7}(\id)$ and hence $\comm^{12}\circ\iota$ to
$\bigl(\partial_{\Syp(2)\to\Sph^7}(\id)\bigr)^{12}$.
The second homotopy deforms the composition of the map
$(a,b) \mapsto [a,b^{12}]$ and the homotopy equivalence~$\iota$ to
$\bigl(\partial_{\Syp(2)\to\Sph^7}(\id)\bigr)^{12}$.
\end{lem}

\begin{cor}
The concatenation of the homotopies in Corollary\,\ref{homone} and Lemma\,\ref{homtwo}
deform $\comm^{12}$ and $(a,b) \mapsto [a,b^{12}]$ both to
$\bigl(\partial_{\Syp(2)\to\Sph^7}(\id)\bigr)^{12}\circ\mu$ and thus conclude Construction\,\ref{homeq}.
\end{cor}

\bigskip

\section{The homotopy between $\kappa\circ\power^2$ and $(N,\Sigma^2\tau)$}
\label{kappatwo}
In this section we perform Construction\,\ref{kappadeform} and construct a homotopy
between the map $\kappa\circ\power^2$ (see subsection\,\ref{geomidentity}) and
a map $(N,\Sigma^2\tau)$. Finally, we perform Construction\,\ref{homlift}.

\subsection{The deformation between $\kappa$ and $h$}
\label{kappatoh}
We first produce a formula for $\kappa$ using complex coordinates. For this we use the
complex cross product definition for the octonionic multiplication on $\C \times \C^3$
from section\,\ref{oct} but not precisely as defined there. For technical purposes
we pull it back by the isometry
\begin{gather*}
  \C \times \C^3 \to \C \times \C^3, \quad
  \bsmat z_0\\ z\esmat \mapsto \bsmat \bar z_0\\ iz\esmat.
\end{gather*}
This saves us some additional deformations. We also need to specify the basis
$e_1,\ldots, e_7$ of subsection\,\ref{geomidentity}. Set
\begin{gather*}
  e_1 = \bbsmat i\\ 0\\ 0\\ 0\eesmat, \quad
  e_2 = \bbsmat 0\\ 1\\ 0\\ 0\eesmat, \quad
  e_3 = \bbsmat 0\\ i\\ 0\\ 0\eesmat, \quad\ldots,
  e_7 = \bbsmat 0\\ 0\\ 0\\ i\eesmat.
\end{gather*}

\begin{lem}
\label{kappacmplx}
With this convention we have $\kappa\bsmat z_0\\ z\esmat = \bigl(
 \kappa_1\bsmat z_0\\ z\esmat, \kappa_2\bsmat z_0\\ z\esmat\bigr)$, where
\begin{gather*}
  \kappa_1\bsmat z_0\\ z \esmat  = \BBsmat i(\abs{z_0}^2-\abs{z}^2)\\
  2 \bar z_0 z_1\\ 2 \bar z_0 z_2\\ 2 \bar z_0 z_3 \EEsmat, \qquad
  \kappa_2\bsmat z_0\\ z \esmat = \BBsmat -2i \Real z_0 z_1\\
    \bar z_0^2 + z_1^2 - \abs{z_2}^2 - \abs{z_3}^2\\
    -2i(\Imag z_1)z_2 + 2i(\Real z_0) \bar z_3\\
    -2i(\Imag z_1)z_3 - 2i(\Real z_0)\bar z_2\EEsmat
\end{gather*}
\end{lem}

\begin{proof}
Straightforward evaluation of
$\bsmat z_0\\ -iz \esmat\cdot e_1\cdot \bsmat \bar z_0\\ iz\esmat$
and
$\bsmat z_0\\ -iz \esmat\cdot e_2\cdot \bsmat \bar z_0\\ iz\esmat$
with the product defined in Theorem\,\ref{octonionicmultiplication}.
\end{proof}

\begin{rem}
It is evident from the above formulas that the columns of $\kappa$ are two
perpendicular variants of the composition of the fibration $\Sph^7\to \CP^3$
and the cut locus collapse $\CP^3\to\Sph^6$.
\end{rem}

Now omit the argument $\bsmat z_0\\ z \esmat$ for a better readability and set
\begin{align*}
  \tilde\kappa_1(s) &=
  \bmat i(\abs{z_0}^2-\abs{z_1}^2 -(\abs{z_2}^2+\abs{z_3}^2)\cos^2s)\\
    2 \bar z_0 z_1\\
    2 z_2(\sin^2s + \bar z_0\cos^2s)-2\bar z_3(1-\bar z_0)\cos s\sin s\\
    2 z_3 (\sin^2s + \bar z_0\cos^2s)+2\bar z_2(1-\bar z_0)\cos s\sin s
  \emat,\\
  \tilde\kappa_2(s) &=
  \bmat -2i \Real z_0 z_1\\
    \bar z_0^2 - z_1^2 - (\abs{z_2}^2 + \abs{z_3}^2)\cos^2s\\
    2 z_2 (\Imag z_1) (\cos s\sin s - i \cos^2s)
      - 2 \bar z_3 ((1-\Real z_0) \cos s\sin s - i(\sin^2s+\Real z_0\cos^2s))\\
    2 z_3(\Imag z_1) (\cos s\sin s - i \cos^2s)
    + 2 \bar z_2 ((1-\Real z_0) \cos s\sin s - i(\sin^2s+\Real z_0\cos^2s))
  \emat.
\end{align*}
Keep first $\kappa_2$ fixed while deforming $\kappa_1$,
then keep $\kappa_1$ fixed while deforming $\kappa_2$, i.e., set
\begin{gather*}
  H_{\kappa}(s) =
  \begin{cases}
    \bigl( \tilde\kappa_1(s), \tilde\kappa_2(0)\bigr), & \text{ if $s \in [0,\tfrac{\pi}{2}]$,}\\
    \bigl( \tilde\kappa_1(\tfrac{\pi}{2}),
      \tilde\kappa_2(s-\tfrac{\pi}{2})\bigr), & \text{ if $s \in [\tfrac{\pi}{2},\pi]$.}
  \end{cases}
\end{gather*}
Define two perpendicular variants $h_1$ and $h_2$ of the Hopf map
$\Sph^3\to \Sph^2$ by
\begin{gather*}
  h_1\bsmat z_0\\ z_1 \esmat =
  \bmat i(\abs{z_0}^2-\abs{z_1}^2)\\ 2\bar z_0 z_1\emat
  \text{ and }
  h_2\bsmat z_0\\ z_1 \esmat =
  \bmat -2i \Real z_0 z_1\\ \bar z_0^2 - z_1^2\emat
\end{gather*}
and extend them to all $\bsmat z_0\\ z_1 \esmat \in \C^2$ by
\begin{gather*}
  h_1\bsmat z_0\\ z_1 \esmat =
    \Abs{\bsmat z_0\\ z_1 \esmat} \, h_1(\bsmat z_0\\ z_1 \esmat/\Abs{\bsmat z_0\\ z_1 \esmat})
  \text{ and }
  h_2\bsmat z_0\\ z_1 \esmat =
    \Abs{\bsmat z_0\\ z_1 \esmat} \, h_2(\bsmat z_0\\ z_1 \esmat/\Abs{\bsmat z_0\\ z_1 \esmat}).
\end{gather*}

\begin{lem}
\label{maindeform}
The homotopy $H_\kappa$ induces a homotopy $\Sph^7\times [0,\pi+1] \to V_{7,2}$
between the map $\kappa$ and the map $h$ given by
\begin{gather*}
  h\bsmat z_0\\ z \esmat =
    \begin{pmatrix}
      h_1\bsmat z_0\\ z_1 \esmat & h_2\bsmat z_0\\ z_1 \esmat\\
      z_2 & i\bar z_3\\ z_3 & -i\bar z_2
    \end{pmatrix}
\end{gather*}
\end{lem}

\begin{proof}
First note that $\tilde\kappa_2(s) = f_2\circ \tilde\kappa_1(s)\circ f_1$, where
$f_1$ and $f_2$ are the two isometries
\begin{gather*}
  f_1\bsmat z_0\\ z \esmat = \BBsmat \Real z_0 + i \Imag z_1\\
  \Real z_1 - i\Imag z_0\\ \Real z_2 + i \Real z_3\\ \x-\Imag z_2 + i\Imag z_3\EEsmat
  \text{ and }
  f_2\bsmat z_0\\ z \esmat = \BBsmat \Real z_0 + i \Real z_1\\
  \Imag z_0 + i\Imag z_1\\ \Imag z_3 + i \Imag z_2\\ \Real z_3 - i\Real z_2\EEsmat.
\end{gather*}
Hence, $\abs{\tilde\kappa_2(s)} = \abs{\tilde\kappa_1(s)}$ and a straightforward
computation shows
\begin{gather*}
  \abs{\tilde\kappa_2(s)} = \abs{\tilde\kappa_1(s)}
    = 1 + (\abs{z_2}^2+\abs{z_3}^2) \sin^2 s.
\end{gather*}
In this and the following computations one can use the fact that the two vectors
$\bsmat z_2\\ z_3 \esmat$ and $\bsmat -\bar z_3\\ \bar z_2 \esmat$ that arise in the
last two components of $\tilde\kappa_1(s)$ and $\tilde\kappa_1(s)$ are
perpendicular with respect to the standard hermitian product on $\C^2$.
Now we verify first that $\tilde\kappa_1(s)$ and $\tilde\kappa_2(0)$
are linearly independent. If $z_2 = z_3 = 0$ this is clear because $h_1$ and $h_2$
are perpendicular. Hence, suppose that $z_2\neq 0$ or $z_3\neq 0$.
The equation
\begin{gather*}
  \tilde\kappa_1(s)
    = \pm \bigl(1 + (\abs{z_2}^2+\abs{z_3}^2) \sin^2 s\bigr)\,\tilde\kappa_2(0)
\end{gather*}
then yields the two equations
\begin{align*}
  \sin^2 s + \bar z_0 \cos^2 s
    &= \mp i \Imag z_1 \bigl(1 + (\abs{z_2}^2+\abs{z_3}^2) \sin^2 s\bigr),\\
  (1-\bar z_0) \cos s \sin s
    &= \mp i \Real z_0 \bigl(1 + (\abs{z_2}^2+\abs{z_3}^2) \sin^2 s\bigr).
\end{align*}
Sorting by real and by imaginary parts one can easily see that there are no solutions.
Next we verify that $\tilde\kappa_1(\tfrac{\pi}{2})$ and $\tilde\kappa_2(s-\tfrac{\pi}{2})$
are linearly independent. As above this is clear if $z_2 = z_3 = 0$.
If $z_2\neq 0$ or $z_3\neq 0$ then the equation
\begin{gather*}
  \bigl(1 + (\abs{z_2}^2+\abs{z_3}^2)\sin^2 s\bigr)\, \tilde\kappa_1(\tfrac{\pi}{2})
    = \pm (1 + \abs{z_2}^2+\abs{z_3}^2)\, \kappa_2(s-\tfrac{\pi}{2})
\end{gather*}
yields the two unsolvable equations
\begin{align*}
  \bigl(1 + (\abs{z_2}^2+\abs{z_3}^2)\sin^2 s\bigr)
     &= \pm (1 + \abs{z_2}^2+\abs{z_3}^2)(\Imag z_1)(\cos s\sin s - i\cos^2 s),\\
  0 &= \pm  ((1-\Real z_0) \cos s\sin s - i(\sin^2s+\Real z_0\cos^2s)).
\end{align*}
Hence, the homotopy $H_{\kappa}$ deforms the map $\kappa$ to the map
\begin{gather*}
  H_{\kappa}(\pi) :\Sph^7 \to V_{7,2}, \quad
  \bmat z_0\\ z \emat \mapsto 
    \bmat
      i(\abs{z_0}^2-\abs{z_1}^2) & -2i \Real z_0 z_1\\
      2\bar z_0 z_1 &  \bar z_0^2 - z_1^2\\
      2 z_2 & 2 i\bar z_3\\ 2z_3 & -2 i\bar z_2
    \emat
\end{gather*}
and the two columns of $\tilde\kappa$ are linearly independent for all $s \in [0,\pi]$.
Now concatenate the homotopy $H_{\kappa}$ with the deformation
\begin{gather*}
  \tfrac{1}{1-s + s\sqrt{\abs{z_0}^2+\abs{z_1}^2}}
    \bmat
      i(\abs{z_0}^2-\abs{z_1}^2) & -2i \Real z_0 z_1\\
      2\bar z_0 z_1 &  \bar z_0^2 - z_1^2\\ 0 & 0\\ 0 & 0
    \emat
    + (2-s) 
    \bmat
      0 & 0\\ 0 & 0\\
      z_2 & i\bar z_3\\ z_3 & - i\bar z_2
    \emat
\end{gather*}
for $s \in [0,1]$ and orthonormalize the two columns for all $s \in [0,\pi+1]$.
\end{proof}

\smallskip

\subsection{The deformation between $h\circ\power^2$ and $(N,\Sigma^2\tau)$}
Recall first that the map $h$ consists of two perpendicular variants of the
fourth suspension $\Sigma^4 h_1$ of the Hopf fibration $h_1:\Sph^3\to \Sph^2$.
Lemma\,\ref{spherenull} provides a concrete null-homotopy of $\Sigma^4 h_1\circ\power^2$.
This null-homotopy can be lifted from $\Sph^6$ to $V_{7,2}$ and we
obtain a homotopy between $h\circ\power^2$ and a map $(N,\sigma)$ for some map
$\sigma:\Sph^7\to \Sph^5$. When we lift the deformation curves horizontally,
however, this map $\sigma$ is not the suspension of a map and we have
no idea how to deform $\sigma\circ\power^2$ to a constant map in a subsequent step.
In order to obtain a suspension we deform the squaring map of $\Sph^7$
to the double suspension of the squaring map of $\Sph^5$
and transfer the problem from $V_{7,2}$ to $V_{5,2}$.

\smallskip
 
For technical purposes we change our coordinates again.
Set $z_2 = a + id$ and $z_3 = c+ib$.
More formally, define an isometry $f_3: \R^4 \to \C^2$ by
\begin{gather*}
  f_3: \bmat a\\ b \\ c\\ d\emat \mapsto \bmat a + i d\\ c + ib\emat
\end{gather*}
and extend it naturally to an isometry $\C^2 \times \R^4 \to \C^2 \times \C^2$.
The map $\tilde h = f_3^{-1}\circ h\circ f_3$ is now given by
\begin{gather*}
  \tilde h : \bmat z_0\\ z_1\\ a\\ b \\ c\\ d\emat \mapsto 
  \begin{pmatrix} h_1\bsmat z_0\\ z_1\esmat & h_2\bsmat z_0\\ z_1\esmat\\
    a & \m b\\ b & -a\\ c & -d \\ d & \m c\end{pmatrix}.
\end{gather*}
Note that $f_3\circ \power^2 = \power^2\circ f_3$ since $f$ does not affect the first four real
coordinates. Hence, $h\circ\power^2 = f_3\circ \tilde h\circ \power^2\circ f_3^{-1}$.

\smallskip

Thus it remains to construct a homotopy between $\tilde h\circ \power^2$
and $(N,\Sigma^2\tilde\tau)$ where $N$ is the constant map to the north pole of $\Sph^7$
and $\tilde\tau$ is a map $\Sph^5 \to \Sph^3$. The map $\tau$ above is then given by
$\tau = f_3\circ \tilde \tau \circ f_3^{-1}$.

\smallskip

We now apply a construction that deforms the squaring map such that the last two
coordinates remain unaltered.

\smallskip

We consider the squaring map $\power^2$ of the general sphere
$\Sph^n \subset \R\times\R^n$. In Cartesian coordinates this map is given by
\begin{gather*}
  \power^2: \bmat x_0\\ \smash[b]{x_1}\\ \vdots \\ x_n\emat \mapsto
  \bmat x_0^2 - x_1^2 -\ldots - x_n^2 \\ 2 \smash[b]{x_0 x_1} \\
  \vdots \\ 2 x_0 x_n \emat.
\end{gather*}
Given a map $\rho : \Sph^{n} \to \Sph^{m}$,
view $\Sph^{n+2}\subset (\R\times \R^n) \times \R^2$ and
$\Sph^{m+2} \subset (\R\times \R^m) \times \R^2$ as double lower suspensions of
$\Sph^{n}$ and $\Sph^{m}$ and let
\begin{gather*}
  \Sigma_2\rho : \Sph^{n+2}\to \Sph^{m+2}, \quad
  \bsmat v\\ w\esmat \mapsto \bsmat \abs{v} \rho(v/\abs{v})\\ w\esmat
\end{gather*}
denote the double lower suspension of $\rho$.
For $x \in \Sph^{n+2}$ and $s \in [0,\frac{\pi}{2}]$ set
\begin{gather*}
  H_{\power^2}(x,s) =
  \bmat
     x_0^2 - x_1^2 - \ldots - x_n^2 -(x_{n+1}^2+x_{n+2}^2)\cos^2s)\\
    2 \bar x_0 x_1\\
    \vdots\\
    2 \bar x_0 x_n\\
    2 x_{n+1}(\sin^2s + x_0\cos^2s) - 2 x_{n+2} (1 - x_0)\cos s\sin s\\
    2 x_{n+2} (\sin^2s + x_0\cos^2s) + 2 x_{n+1}(1 - x_0)\cos s\sin s
  \emat.
\end{gather*}

\begin{lem}
\label{squaringdeform}
The homotopy $H_{\power^2}$ induces a homotopy
$\Sph^{n+2}\times [0,\frac{\pi}{2}] \to \Sph^{n+2}$ between the squaring map
$\power^2$ of $\Sph^{n+2}$ and the double lower suspension
$\Sigma_2\power^2$ of the squaring map of $\Sph^n$.
\end{lem}
\begin{proof}
Contained in the proof of Lemma\,\ref{maindeform}.
\end{proof}

\smallskip

Thus it remains to construct a homotopy between $\tilde h \circ \Sigma_2\power^2$
and $(N,\Sigma^2\tilde\tau)$.

\smallskip

We now apply the following construction: Let
\begin{gather*}
  (\alpha,\beta): \Sph^n \times [0,1] \to V_{m,2}\subset \Sph^{m-1}\times\Sph^{m-1}
\end{gather*}
be a homotopy between the two maps $(\alpha_0,\beta_0)$ and $(\alpha_1,\beta_1)$.
Extend $\alpha$ and $\beta$ to maps $\R^{n+1}\times [0,1] \to \R^{m}$ by setting
$\alpha(v,s) = \abs{v}\alpha(\frac{v}{\abs{v}},s)$ and
$\beta(v,s) = \abs{v}\beta(\frac{v}{\abs{v}},s)$. 
Now define the two homotopies
$\tilde \alpha, \tilde\beta: \Sph^{n+2} \times [0,1] \to \Sph^{m+1}\subset \R^{m+2}$
by setting
\begin{gather*}
  \tilde\alpha(\Bsmat v\\ c\\ d\Esmat, s) = \Bsmat \alpha(v,s)\\ c \\ d\Esmat
  \quad\text{ and }\quad
  \tilde\beta(\Bsmat v\\ c\\ d\Esmat, s) = \Bsmat \beta(v,s)\\ -d \\ c\Esmat.
\end{gather*}
It is evident that $\tilde\alpha$ and $\tilde\beta$ are always perpendicular.
The following statement is now obvious.

\begin{lem}
\label{stiefel}
The assignment $(\alpha,\beta) \mapsto (\tilde\alpha,\tilde\beta)$ induces a
homomorphism $\pi_n(V_{m,2}) \to \pi_{n+2}(V_{m+2,2})$.
\end{lem}

\smallskip

Thus it remains to construct a homotopy between $\tilde h_{5,2}\circ\power^2$ and the map
$(N,\tilde\tau)$, where $\tilde h_{5,2} : \Sph^5 \to V_{5,2}$ is the map
\begin{gather*}
  \tilde h_{5,2} : \bmat z_0\\ z_1\\ a\\ b \emat \mapsto
  \begin{pmatrix} h_1\bsmat z_0\\ z_1\esmat & h_2\bsmat z_0\\ z_1\esmat\\
    a & \m b\\ b & -a\end{pmatrix}.
\end{gather*}
This is because the concrete homomorphism of Lemma\,\ref{stiefel} yields a deformation
\begin{gather*}
  \tilde h \circ \Sigma_2\power^2\bmat z_0\\ z_1 \\ a \\ b\\ c \\ d & \emat
  \sim
  \begin{pmatrix} i\Abs{\bmat z_0\\ z_1\emat} & \m 0 \\
    0 & \tilde \tau \bmat z_0\\ z_1 \\ a \\ b \emat \\ c & -d \\ d & \m c
  \end{pmatrix}
\end{gather*}
and the latter map can be deformed by
\begin{gather*}
  \begin{pmatrix} i\cos (s\arccos(\Abs{\bmat z_0\\ z_1\emat})) & \m 0 \\
    0 & \m \tilde \tau \\
    \frac{c}{\sqrt{a^2+b^2+c^2+d^2}}\sin(s\arccos(\Abs{\bmat z_0\\ z_1\emat})) & -d \\
    \frac{d}{\sqrt{a^2+b^2+c^2+d^2}}\sin(s\arccos(\Abs{\bmat z_0\\ z_1\emat})) & \m c
  \end{pmatrix}
  \quad\text{ to }\quad
  \begin{pmatrix} i & \m 0 \\
    0 & \m \tilde \tau \\ 0 & -d \\ 0 & \m c
  \end{pmatrix}.
\end{gather*}
and, finally, by applying some evident rotations in the domain of definition and
in the target domain, to the map $(N,\Sigma^2\tilde\tau)$.

\smallskip

In the rest of this subsection we construct the remaining homotopy between
$\tilde h_{5,2}\circ\power^2$ and $(N,\tilde\tau)$. It is easy to deform the identity
on $\C^2\times \R^2$ to the isometry
\begin{gather*}
  f_5:\bmat z_0\\ z_1\\ a\\ b \emat \mapsto \bmat a+ ib\\ z_1\\ \Real z_0\\ \Imag z_0 \emat
\end{gather*}
and the identity on $i\R\times\C\times \R^2$ to the isometry
\begin{gather*}
  f_6:\bmat iy\\ w\\ a\\ b \emat \mapsto \bmat ia\\ b + iy\\  \Real w\\ \Imag w \emat.
\end{gather*}
The map $\tilde h_{5,2}$ is thus homotopic to the map
$\hat h_{5,2} = f_6\circ \tilde h_{5,2}\circ f_5$
and, accordingly, $\tilde h_{5,2}\circ \power^2$ is homotopic to $\hat h_{5,2}\circ\power^2$.
The first column of the map $\hat h_{5,2}$ is now nothing but the second suspension
$\Sigma^2 h_1$ of $h_1$. By Lemma\,\ref{spherecommute} we have
\begin{gather*}
  \Sigma^2 h_1\circ \power^2 = \power^2 \circ \Sigma^2 h_1
\end{gather*}
and an explicit null-homotopy of $\power^2 \circ \Sigma^2 h_1$
is supplied in Lemma\,\ref{spherenull}.
Starting with our map $\hat h_{5,2}\circ\power^2$ this null-homotopy
$\Sph^5 \times [0,1] \to \Sph^4$ can be lifted horizontally to the
Stiefel manifold $V_{5,2}$. This yields a homotopy
\begin{gather*}
   \hat h_{5,2} \circ \power^2\sim \begin{pmatrix} i & 0 \\ 0 & \tilde\tau
    \end{pmatrix}
\end{gather*}
with a map $\tilde\tau: \Sph^5 \to \Sph^3$. It is not difficult to solve the ODE
$\delta' (s) = -\langle \delta(s),\gamma'(s)\rangle\gamma(s)$ for the horizontal lift
$(\gamma(s),\delta(s))$ of the curves $\gamma(s) = H_1(\ldots,s)$  and
$\gamma(s) = H_2(\ldots,s)$ of Lemma\,\ref{spherenull} explicitly
and thus to write down a closed formula for the map $\tilde\tau$.
We omit to present this lengthy formula since it is completely irrelevant for the following steps.

\smallskip

\subsection{Performance of Construction\,\ref{homlift}}
Now let $\alpha_s$ denote the deformation between the two maps
$\alpha_0 = \kappa\circ\power^2$ and $\alpha_1 = (N,\Sigma^2\tau)$
constructed in the previous two subsections. Note that the image of the
south pole $\bmat -1\\ 0\emat$ under $\alpha_s$ varies with $s$.
Thus, the deformation $\partial_{\Gtwo\to V_{7,2}}(\alpha_s)$
does not take place in just one $\Sph^3$-fiber.
There are several elementary ways to fix this issue.
We choose one that fits perfectly to the subsequent homotopie $H_{\SU(3)\to\Sph^5}$
given in Theorem\,\ref{suhom}. Let $v_0 = \bmat i\\ 0\emat \in \Sph^7\subset \C\times\C^3$
and $\gamma_{v_0}(t) = \bmat e^{it}\\ 0\emat$. By Lemma\,\ref{kappacmplx} we have
\begin{gather*}
  \alpha_0\circ\gamma_{v_0}(t) = \bmat i & 0\\ 0 & \m e^{-2it}\\ 0 & 0\\ 0 & 0\emat
\end{gather*}
and, by the definition of the suspension, we have
\begin{gather*}
  \alpha_1\circ\gamma_{v_0}(t) = \bmat i & 0\\ 0 & \m e^{it}\\ 0 & 0\\ 0 & 0\emat.
\end{gather*}
In both cases, the horizontal lifts of these curves are contained in $\SU(3) \subset \Gtwo$,
since the first column is constant. It is easily verified that
\begin{gather*}
  \widetilde{\alpha_0\circ\gamma_{v_0}}(t) =
  \bmat e^{-2it} & 0 & 0 \\ \x 0 & \m e^{it} & 0\\ \x 0 & 0 & \m e^{it}\emat
\end{gather*}
and
\begin{gather*}
  \widetilde{\alpha_1\circ\gamma_{v_0}}(t) =
  \bmat e^{it} & 0 & 0 \\ 0 & \m e^{-it/2} & 0\\ 0 & 0 & \m e^{-it/2}\emat.
\end{gather*}
Hence,
\begin{gather*}
  \partial_{\Gtwo\to V_{7,2}}(\alpha_0)(v_0) = \widetilde{\alpha_0\circ\gamma_{v_0}}(\pi)
  = \bmat 1 & \m 0 & \m 0\\ 0 & -1 & \m 0\\ 0 & \m 0 & -1\emat
\end{gather*}
and
\begin{gather*}
  \partial_{\Gtwo\to V_{7,2}}(\alpha_1)(v_0) = \widetilde{\alpha_1\circ\gamma_{v_0}}(\pi)
  = \bmat -1 &  \m 0 & \m 0\\ \m 0 & -i & \m 0\\ \m 0 & \m 0 & -i\emat.
\end{gather*}
We now consider the homotopy
\begin{gather*}
  \partial_{\Gtwo\to V_{7,2}}(\alpha_0)(v_0)
    \cdot \bigl(\partial_{\Gtwo\to V_{7,2}}(\alpha_s)(v_0)\bigr)^{-1}
    \cdot \partial_{\Gtwo\to V_{7,2}}(\alpha_s).
\end{gather*}
This homotopy deforms $\partial_{\Gtwo\to V_{7,2}}(\kappa\circ\power^2)$ to
\begin{gather*}
  \bmat -1 &  \m 0 & \m 0\\ 0 & -i & \m 0\\ 0 & \m 0 & -i\emat\cdot
    \partial_{\Gtwo\to V_{7,2}}(N,\Sigma^2\tau)
  =  \bmat -1 &  \m 0 & \m 0\\ \m 0 & -i & \m 0\\ \m 0 & \m 0 & -i\emat\cdot
    \partial_{\SU(3)\to\Sph^5}(\Sigma\tau)
\end{gather*}
within the fixed $\Sph^3$-fiber over $(e_1,e_2) \in V_{7,2}$.
The last map fits perfectly to Theorem\,\ref{suhom}.

\bigskip

\section{Nontrivial maps from $\Sph^7$ to $\Syp(2)$ and exotic actions on $\Sph^7\times\Sph^3$}
\label{exotic}

\subsection{Nontrivial maps from $\Sph^7$ to $\Syp(2)$}
Consider $\Sph^7$ with north pole $N = (1,0,\ldots,0)$. As in section\,\ref{lifts} let $\gamma_v$
denote the geodesic of $\Sph^7$ (with respect to the standard metric on 	$\Sph^7$) with
$\gamma_v(0) = N$ and $\dot\gamma_v(0) = v$.
Let $D^7(\pi)$ denote the disk in the tangent space $T_N\Sph^7$ with radius $\pi$.
For a fixed but arbitrary $j\in \Z$ lift the geodesics $t \mapsto \gamma_v(12jt)$
horizontally with respect to the fibration $\Syp(2) \to \Sph^7$. This yields a map
$\xi_j : D^7(\pi) \to \Syp(2)$ with
\begin{gather*}
  \xi_{j\big\vert \Sph^6(\pi)} = \partial_{\Syp(2)\to \Sph^7}(\power^{12j}).
\end{gather*}
On the other hand, the null-homotopy of
\begin{gather*}
  \bigl(\partial_{\Syp(2)\to \Sph^7}(\id)\bigr)^{12j} = \partial_{\Syp(2)\to \Sph^7}(\power^{12j})
\end{gather*}
(see Corollary\,\ref{powerid} for this identity) provides us with a map
$\zeta_j : D^7(\pi) \to \Sph^3$ with 
\begin{gather*}
  \zeta_{j\big\vert \Sph^6(\pi)} = \partial_{\Syp(2)\to \Sph^7}(\power^{12j}).
\end{gather*}
We now define a map $\chi_j : D^7(\pi) \to \Syp(2)$ by
\begin{gather*}
  \chi_j(v) = \xi_j(v) \cdot \bmat 1 & 0 \\ 0 & \zeta_j^{-1}\emat.
\end{gather*}

\begin{lem}
The map $\chi_j$ induces a map $\Sph^7 \to \Syp(2)$ whose first column is the $12j$-th power
$\power^{12j}$ of $\Sph^7$. Hence, the map $\chi_j$ represents the $j$-th homotopy
class in $\pi_7(\Syp(2)) \approx \Z$.
\end{lem}

\begin{proof}
By definition the map $\chi_j$ evaluates constantly to the north pole
on the boundary of $D^7(\pi)$. Hence, we get a map from $\Sph^7$ to $\Syp(2)$
whose first column is $\power^{12j}$, a map of degree $12j$. The last claim follows now
from the relevant part of the exact homotopy sequence of the bundle $\Syp(2)\to\Sph^7$.
\end{proof}

\subsection{Exotic actions on $\Sph^7\times\Sph^3$}
We now combine the maps $\chi_j$ above with the generalized Gromoll-Meyer
construction in \cite{dpr}. Let $E_n^{10}$ denote the $\Sph^3$-principal bundle obtained
by pulling back $\Syp(2)\to \Sph^7$ by the $n$-th power $\power^n$ of~$\Sph^7$:
\begin{gather*}
\begin{CD}
E_n^{10} @>>> \Syp(2) \\
@VVV @VVV \\
\Sph^7 @>{\power^n}>> \Sph^7.
\end{CD}
\end{gather*}
Explicitly,
\begin{gather*}
  E_n^{10} := \bigl\{ (u,v) \in \Sph^7\times \Sph^7\;\big\vert\;
    \llangle \power^n(u), v\rrangle = 0 \bigr\}
\end{gather*}
where $\llangle\,\cdot\,,\,\cdot\,\rrangle$ denotes the standard Hermitian inner product on
the quaternionic vector space $\H^2$. The total spaces $E_n^{10}$ come equipped with
a free action of the unit quaternions:
\begin{gather*}
  \Sph^3 \times E_n^{10} \to E_n^{10},\quad
  q \star (u,v) = (q u \bar q, q v).
\end{gather*}
Here, $qu\bar q$ means that the two quaternionic components of $u$ are
simultaneously conjugated by $q \in \Sph^3$.
The quotient of $E_n^{10}$ by the free $\star$-action is a smooth manifold
\begin{gather*}
  \Sigma^7_n := E_n^{10}/\Sph^3.
\end{gather*}

\begin{thm}[\cite{dpr}]
The differentiable manifold $\Sigma^7_n$ is a homotopy sphere and represents
the $(n\mod 28)$-th element in $\Theta_7\approx\Z_{28}$.
\end{thm}

Now the map $\chi_j: \Sph^7 \to \Syp(2)$ from the previous subsection supplies us
immediately with the section $\Sph^7 \to E^{10}_{12j}$, $u \mapsto (u,\chi_{j,2}(u))$
of the principal bundle $E^{10}_{12j} \to \Sph^7$. Here, $\chi_{j,2}(u)$ means the
second column of $\chi_j$. We obtain the trivialization
\begin{gather*}
  \Sph^7 \times \Sph^3 \to E^{10}_{12j}, \quad
  (u,r) \mapsto \bigl(u, \chi_{j,2}(u) r\bigr)
\end{gather*}
with inverse
\begin{gather*}
  E^{10}_{12j} \to \Sph^7 \times \Sph^3, \quad
  (u,v) \mapsto \bigl(u, \llangle \chi_{j,2}(u),v\rrangle\bigr).
\end{gather*}
This trivialization can be used to transfer the Gromoll-Meyer action from $E^{10}_{12j}$
to the product $\Sph^7\times\Sph^3$. This way we obtain the formula of
Theorem\,\ref{exoticaction} from the introduction.

\bigskip

\section*{Acknowledgements}
The author would like to thank A.~Rigas for his constant encouragement and advice throughout
this project and Carlos Duran for several valuable discussions.
The author was funded by a Heisenberg fellowship of Deutsche Forschungsgemeinschaft in the
years 2006--2008 and supported by the DFG priority programm SPP~1154.

\bigskip

%
%

%

\end{document}